\documentclass[12pt]{amsart}
\usepackage{amssymb,amsmath,amsthm,amsfonts}
\usepackage[mathscr]{euscript}
\usepackage{dsfont}
\usepackage{graphicx}
\usepackage{mathrsfs}
\usepackage{esint}
\usepackage{xcolor}
\usepackage[usenames,dvipsnames]{pstricks}
\usepackage{epsfig}
\usepackage{pst-grad} 
\usepackage{pst-plot}
\usepackage{color}
\definecolor{citation}{rgb}{0.2,0.5,0.2}
\definecolor{formula}{rgb}{0.1,0.2,0.5}
\definecolor{url}{rgb}{0,0.2,0.7}
\usepackage[colorlinks=true,linkcolor=formula,urlcolor=url,citecolor=citation]{hyperref}
\allowdisplaybreaks[4]
\usepackage{color}
\definecolor{citation}{rgb}{0.2,0.5,0.2}
\definecolor{formula}{rgb}{0.1,0.2,0.5}
\definecolor{url}{rgb}{0,0.2,0.7}
\usepackage[colorlinks=true,linkcolor=formula,urlcolor=url,citecolor=citation]{hyperref}
\allowdisplaybreaks[4]

\def\ds{\displaystyle}

\def\supp{\mathrm{supp}}
\def\a{{\mathcal{A}}}
\def\l{{\mathcal{L}}}
\newcommand{\R}{\mathbb{R}}

\newcommand{\N}{\mathbb{N}}
\newcommand{\D}{\mathrm{d}}
\newcommand{\e}{\mathcal{E}}
\newcommand{\eps}{\varepsilon}
\renewcommand{\tilde}{\widetilde}
\renewcommand{\hat}{\widehat}
\renewcommand{\epsilon}{\varepsilon}
\newcommand{\oplb}[2]{\l_{_{#2}}[{#1}]}
\newcommand{\p}[1]{\mathcal{P}_{{#1}}}

\newcommand{\SE}{\setcounter{equation}{0} \section}
\newcommand{\baa}{\begin{array}}
\newcommand{\eaa}{\end{array}}
\newcommand{\ba}{\begin{eqnarray}}
\newcommand{\ea}{\end{eqnarray}}
\newcommand{\be}{\begin{equation}}
\newcommand{\ee}{\end{equation}}
\newtheorem{theorem}{Theorem}[section]
\newtheorem{lemma}[theorem]{Lemma}

\newtheorem{prop}[theorem]{Proposition}

\theoremstyle{definition}
\newtheorem{definition}[theorem]{Definition}
\theoremstyle{remark}
\newtheorem{remark}[theorem]{Remark}
\theoremstyle{remark}

\renewcommand{\le}{\leqslant}
\renewcommand{\leq}{\leqslant}
\renewcommand{\ge}{\geqslant}
\renewcommand{\geq}{\geqslant}

\newlength{\defbaselineskip}
\begin{document}

\title[]{Liouville type results for a nonlocal obstacle problem}$\,$\thanks{This work has been carried out in the framework of Archim\`ede Labex (ANR-11-LABX-0033) and of the A*MIDEX project (ANR-11-IDEX-0001-02), funded by the ``Investissements d'Avenir" French Government program managed by the French National Research Agency (ANR). The research leading to these results has also received funding from the
European Research Council under the European Union's Seventh Framework Programme (FP/2007-2013) ERC Grant Agreement n.~321186~- ReaDi~- Reaction-Diffusion Equations, Propa\-gation and Modelling,
and from the ANR DEFI project NONLOCAL (ANR-14-CE25-0013) and the ANR JCJC project MODEVOL (ANR-13-JS01-0009).
Additional support came from the Australian Research Council under
the Discovery Project NEW (Nonlocal Equations at Work) DP-170104880.}

\author[Julien Brasseur]{Julien Brasseur}
\address{BioSP, INRA, 84914, Avignon, France, and Aix Marseille Univ, CNRS, Centrale Marseille, I2M, Marseille, France}
\author[J\'er\^{o}me Coville]{J\'er\^ome Coville}
\address{BioSP, INRA, 84914, Avignon, France}
\author[Fran\c{c}ois Hamel]{Fran{\c{c}}ois Hamel}
\address{Aix Marseille Univ, CNRS, Centrale Marseille, I2M, Marseille, France}
\author[Enrico Valdinoci]{Enrico Valdinoci}
\address{School of Mathematics and Statistics, University of Melbourne, 813 Swanston Street, Parkville VIC 3010, Australia, and Dipartimento di Matematica, Universit\`a degli Studi di Milano, Via Saldini 50, 20133 Milan, Italy}

\begin{abstract}
This paper is concerned with qualitative properties of solutions to nonlocal reaction-diffusion equations of the form
$$ \int_{\mathbb{R}^N\setminus K} J(x-y)\,\big( u(y)-u(x) \big)\,\D y+f(u(x))=0, \quad x\in\R^N\setminus K,$$
set in a perforated open set $\mathbb{R}^N\setminus K$, where $K\subset\mathbb{R}^N$ is a bounded compact ``obstacle" and $f$ is a bistable nonlinearity. When $K$ is convex, we prove some Liouville-type results for solutions satisfying some asymptotic limiting conditions at infinity. We also establish a robustness result, assuming slightly relaxed conditions on $K$.
\end{abstract}

\maketitle

\tableofcontents


\SE{Introduction}

A classical topic in applied analysis consists
in the study of diffusive processes
in media with an obstacle: roughly speaking,
a dispersal follows a Brownian motion in an environment
that possess an inaccessible region. At the level
of partial differential equations, this translates into
a reaction/diffusion equation that is defined
outside a set~$K$,
which acts as an 
impenetrable
obstacle and along which Neumann conditions are prescribed.

One of the cornerstones in the study of these processes
lies in suitable rigidity results of Liouville-type,
which allow the classification of stationary solutions,
at least under some geometric assumption on the obstacle~$K$.

In this paper, we will study a nonlocal version
of a diffusion equation and provide a series of
Liouville-type results (whose
precise statements will be given
in Section~\ref{MAIN RES}). Not only
the results obtained have a theoretical interest in the
development of the theory of nonlocal equations, but they
also possess several potential applications (especially
in mathematical biology,
where the dispersal of biological populations
often presents nonlocal features, see e.g.
formula~(1) in~\cite{COSNER}, or in \cite{BR}).

Concretely, we will suppose that the diffusion operator
arises by convolution with an integrable kernel
and we will show that solutions of bistable
stationary equations with fixed behavior at infinity are necessarily
constant, at least when the obstacle is convex
or ``close to being convex'' (we also observe
that similar rigidity results
do not hold in general for nonconvex obstacles).

Interestingly, in the nonlocal case, the boundary conditions
along the obstacle do not need to be prescribed a priori
(differently from the classical case).

In addition, the nonlocal operator that we consider here
is not ``regularizing'', so some care is needed in our case
to deal with a possible
lack of regularity of the solutions.

We now provide the detailed mathematical description of the problem that we take into account.


\subsection{A nonlocal obstacle problem}

Throughout this paper, $K$ denotes a compact set of $\R^N$ with $N\ge2$, and $\left|\cdot\right|$ denotes the Euclidean norm in $\R^N$. We are interested in qualitative properties of bounded solutions to the following nonlocal semilinear equation
\begin{align}
Lu+f(u)=0 \quad{\mbox{ in }}\R^N\setminus K, \label{EQt}
\end{align}
where $L$ is the nonlocal diffusion operator given by
\begin{equation}\label{DEF:L}
Lu(x):=\int_{\R^N\setminus K} J(x-y)\,\big( u(y)-u(x) \big)\,\D y.
\end{equation}
The kernel $J\in L^1(\R^N)$ is a radially symmetric non-negative function with unit mass and $f$ is a $C^1$ ``bistable" nonlinearity (precise assumptions on $f$ and $J$ will be given later on).

This problem may be thought of (see the next page for more explanations) as a nonlocal version of the following problem
\be\label{eqlocale}\left\{\begin{array}{rl}
\Delta u+f(u)=0 & \text{in }\R^N\setminus K,\vspace{3pt}\\
\nabla u\cdot\nu=0 & \text{on }\partial K,
\end{array}\right.
\ee
where $\nu$ is the outward unit vector normal to $K$, assuming for~\eqref{eqlocale} that $K$ is smooth enough. For problem~\eqref{eqlocale} with the local diffusion operator $\Delta u$, it was shown in~\cite{BHM} that there
exist a time-global classical solution $u(t,x)$ to the parabolic problem
\be\label{EQ:BHM}\left\{\baa{rcll}
\ds\frac{\partial u}{\partial t} & \!\!=\!\! & \Delta u+f(u) & \text{in }\R\times\overline{\R^N\setminus K},\vspace{3pt}\\
\nabla u\cdot\nu & \!\!=\!\! & 0 & \text{on }\R\times\partial K\eaa\right.
\ee
satisfying $0<u(t,x)<1$ for all $(t,x)\in\R\times\overline{\R^N\setminus K}$,
and a classical solution $u_\infty(x)$ to the elliptic problem
\begin{align}
\left\{
\begin{array}{rl}
\Delta u_\infty+f(u_\infty)=0 & \text{in }\overline{\R^N\setminus K},\vspace{3pt}\\
\nabla u_\infty\cdot\nu=0 & \text{on }\partial K, \vspace{3pt}\\
0\le u_\infty\leq 1 & \text{in }\overline{\R^N\setminus K}, \vspace{3pt}\\
u_\infty(x)\to1 & \text{as }|x|\to+\infty.
\end{array} \label{EQ:BHM0}
\right.
\end{align}
The function $u_{\infty}$ is a stationary solution of~\eqref{EQ:BHM} and it is actually obtained as the large time limit of $u(t,x)$, in the sense that $u(t,x)\to u_\infty(x)$ as $t\to+\infty$ locally uniformly in $x\in\overline{\R^N\setminus K}$. Under some geometric conditions on $K$ (e.g. if $K$ is starshaped or directionally convex, see~\cite{BHM} for precise assumptions) it is shown in~\cite[Theorems~6.1 and~6.4]{BHM} that solutions to~\eqref{EQ:BHM0} are actually identically equal to~$1$ in the whole set $\overline{\R^N\setminus K}$. This Liouville property shows that the solutions $u(t,x)$ of~\eqref{EQ:BHM} constructed in~\cite{BHM} then satisfy
\begin{align}
u(t,x)\underset{t\to+\infty}{\longrightarrow} 1 \qquad{\mbox{ locally uniformly in }}x\in\overline{\R^N\setminus K}. \label{longtime}
\end{align}

To some extent, this result can be given an ecological interpretation. Consider a population with trajectories describing a Brownian motion in an environment consisting of the whole space~$\R^N$ with a compact obstacle $K$, and suppose that $f$ represents the demographic rate of the population. Then, the solution $u(t,x)$ to~\eqref{EQ:BHM} can be understood as the density of the population at time $t$ and location $x$. In this context,~\eqref{longtime} means that, at large time, the population tends to occupy the whole space.

Assuming now that the trajectories follow, say, a compound Poisson process, then the diffusion phenomena are better described by a convolution-type operator such as~\eqref{DEF:L}. The reaction-diffusion equation $\frac{\partial u}{\partial t}=\Delta u+f(u)$ is then replaced by the equation
$$\frac{\partial u}{\partial t}=Lu+f(u)$$
with the nonlocal dispersion operator $L$, see \cite{Fife,Hutson}. In this paper, we deal with qualitative properties of the stationary
solutions of equation~\eqref{EQt}, together with some asymptotic limiting conditions at infinity similar to those appearing in~\eqref{EQ:BHM0}. Namely, we will be mainly concerned with  solutions of
\begin{align}
\left\{
\begin{array}{rl}
Lu+f(u)=0 & \text{in }\R^N\setminus K,\vspace{3pt}\\
0\leq u\leq 1 & \text{in }\R^N\setminus K, \vspace{3pt}\\
u(x)\to1 & \text{as }|x|\to+\infty.
\end{array} \label{EQ:BHM1}
\right.
\end{align}

It is expected that~\eqref{EQ:BHM0} and~\eqref{EQ:BHM1} share some common properties. One of the goals of the present paper is, as for~\eqref{EQ:BHM0}, to find some geometric conditions on $K$ which guarantee that the solutions $u$ to~\eqref{EQ:BHM1} are identically equal to $1$. Moreover, as in~\cite{Bouhours} for~\eqref{EQ:BHM0}, we will also show the robustness of the Liouville type results for~\eqref{EQ:BHM1}.

We notice however that, whereas the solutions of~\eqref{EQ:BHM0} are automatically classical $C^2$ solutions in~$\overline{\R^N\setminus K}$ if the boundary~$\partial K$ is smooth enough (by standard interior and boundary elliptic estimates),
there is in general no smoothing effect for the nonlocal problems~\eqref{EQt} or~\eqref{EQ:BHM1}. The solutions $u$ may even not be continuous in general. Yet some regularity results (uniform or H\"older continuity) will be shown here under additional assumptions on the data $J$ and $f$. Actually, one of the difficulties and novelties of this paper, as compared to~\cite{BHM}, is to deal with this a priori
{\em lack of regularity} in general.

We observe also that, in~\eqref{EQt} or~\eqref{EQ:BHM1}, we {\em
do not ask for any additional boundary condition on $\partial K$}. To understand why this is so, let us make some heuristic comments. First of all, the most intuitive nonlocal counterpart of~\eqref{EQ:BHM0} would be to replace $\Delta u$ in~\eqref{EQ:BHM0} by $\tilde{L}_\eps u$ with $\eps>0$ small, where
$$\tilde{L}_\eps u(x):=\frac{1}{\beta\eps^2}\int_{\R^N\setminus K}\tilde{J}_\eps(x-y)\big(u(y)-u(x)\big)\mathrm{d}y$$
and $\tilde{J}_\eps(z)=\eps^{-N}\tilde{J}(\eps^{-1}z)$, $\tilde{J}$ being a radially symmetric kernel with $\beta=(2N)^{-1}\int_{\R^N}\tilde{J}(z)\,|z|^2\,\D z$. In other words, the nonlocal dispersion operator $Lu$ in~\eqref{EQt} or~\eqref{EQ:BHM1} would be replaced by $\tilde{L}_\eps u$ and the kernel $J$ would be given by $(\beta\eps^2)^{-1}\tilde{J}_\eps$. Furthermore, using for example~\cite{BBM}, the associated energy of~\eqref{EQ:BHM1}, with $\tilde{L}_\eps$ in place of $L$, can be thought of as an approximation of that of~\eqref{EQ:BHM0} (see \cite{Andreu} and  also~\cite{BCV,BR,Hutson} where similar quantities are considered in a biological framework). Now, to see how the Neumann boundary condition in~\eqref{EQ:BHM0} can be recovered from~\eqref{EQt} or~\eqref{EQ:BHM1} with~$\tilde{L}_\eps$ as~$\eps\to0^+$, let us consider for simplicity the case where $\partial K$ is of class $C^1$ with unit normal~$\nu$ and the bounded function~$u$ is of class $C^1(\overline{\R^N\setminus K})$ and is extended as a $C^1(\R^N)$ function still denoted by $u$. Formula~\eqref{EQt} then also holds by continuity in $\overline{\R^N\setminus K}$ and, for every~$x,y\in\overline{\R^N\setminus K}$,
there exists a point $c_{x,y}\in[x,y]$ such that $ u(y)-u(x)=\nabla u(c_{x,y})\cdot(y-x)$. It follows that, for every $x\in\overline{\R^N\setminus K}$,
$$-f(u(x))=\frac{1}{\beta\eps}\int_{\R^N\setminus K}\hat{J}_{\eps}(x-y)\,\nabla u(c_{x,y})\cdot \frac{y-x}{|y-x|}\,\mathrm{d}y,$$
where $\hat{J}_\eps(z)=\eps^{-N}\hat{J}(\eps^{-1}z)$ and $\hat{J}(z)=\tilde{J}(z)|z|$. Then, for all $x\in\partial K$, a formal computation leads to
$$\gamma\,\nabla u(x)\cdot\nu=\lim_{\eps\to0^+}\int_{\R^N\setminus K}\hat{J}_{\eps}(x-y)\,\nabla u(c_{x,y})\cdot \frac{y-x}{|y-x|}\,\mathrm{d}y=\lim_{\eps\to0^+}\!\big(\!-\eps\beta f(u(x))\big)=0,$$
where $\gamma=(1/2)\int_{\R^N}\tilde{J}(z)\,|z_1|\,\D z>0$. Hence, $\nabla u\cdot\nu=0$ on $\partial K$ and~\eqref{EQ:BHM1} is then a reasonable nonlocal counterpart for~\eqref{EQ:BHM0}. The above calculation justifies, at least formally, why no additional boundary condition on $\partial K$ is required in~\eqref{EQt} or~\eqref{EQ:BHM1}.


\subsection{General assumptions, notations and definitions}

Let us now specify the detailed assumptions made throughout the paper. As already mentioned above, we suppose that $f$ is of ``bistable'' type and $J$ is a radially symmetric kernel. More precisely, we will assume that
\be\label{C1}
f\in C^1([0,1]),\ \ f(0)\ge0,\ \ f'(1)<0,
\ee
\be\label{C2}\left\{\baa{l}
J\in L^1(\R^N)\hbox{ is a non-negative, radially symmetric kernel with unit mass},\vspace{3pt}\\
\hbox{there are }0\le r_1<r_2\hbox{ such that }J(x)>0\hbox{ for a.e. }x\hbox{ with }r_1<|x|<r_2,\eaa\right.
\ee
and there exists a function $\phi\in C(\R)$ satisfying
\be\label{C3}
\left\{
\begin{array}{l}
J_1\ast \phi-\phi+f(\phi)\geq 0\text{ in }\R, \vspace{3pt}\\
\phi\hbox{ is increasing in }\R,\ \ \phi(-\infty)=0,\ \ \phi(+\infty)=1,
\end{array}
\right.
\ee
where $J_1\in L^1(\R)$ is the non-negative even function with unit mass given for a.e. $x\in\R$ by
$$J_1(x):=\int_{\R^{N-1}} J(x,y_2,\cdots,y_N)\,\D y_2\cdots\D y_N.$$
We notice that, in addition to the first property in~\eqref{C2}, the second one is immediately fulfilled if $J$ is assumed to be continuous. 
Moreover, we notice that condition~\eqref{C3} implies immediately that $0<\phi<1$ in $\R$. As is well-known (see e.g.~\cite{Bates,Coville}), condition~\eqref{C3} is satisfied if, in addition to~\eqref{C1} and~\eqref{C2}, the following assumptions are made on~$f$ and~$J$:
\be\label{C4}
\left\{
\begin{array}{l}
\exists\,\theta\in(0,1),\ \ f(0)=f(\theta)=f(1)=0,\ \ f<0\hbox{ in }(0,\theta),\ \ f>0\hbox{ in }(\theta,1),\vspace{3pt}\\
\displaystyle\int_{0}^1f>0,\ \ f'(0)<0,\ \ f'(\theta)>0,\ \ f'(1)<0,\ \ f'<1\text{ in }[0,1],\vspace{3pt}\\
\displaystyle\int_{\R}J_1(x)|x|\,\D x<+\infty~~~~\text{  and  }~~~~J\in W^{1,1}(\R^N).
\end{array}
\right.
\ee

Let us also list in this subsection a few notations and definitions used in the paper:
\begin{tabular}{rl}
$|E|$ & is the Lebesgue measure of the measurable set $E$; \\
$\mathds{1}_E$ & is the characteristic function of the set $E$; \\
$B_R$ & is the open Euclidean ball of radius $R>0$ centered at the origin; \\
$B_R(x)$ & is the open Euclidean ball of radius $R>0$ centered at $x\in\R^N$; \\
$\mathcal{A}(R_1,R_2)$ & is the open annulus $B_{R_2}\setminus\overline{B_{R_1}}$ for $0\le R_1<R_2$, by setting $\overline{B_0}=\{0\}$; \\
$\mathcal{A}(x,R_1,R_2)$ & is the open annulus $x+\mathcal{A}(R_1,R_2)$; \\
$g\ast h$ & is the convolution of $g$ and $h$; \\
$g^+$ & is the positive part of $g$, i.e. $g^+:=\max\{0,g\}$.
\end{tabular} \\

Given $\Omega\subset\R^N$ and $p\in[1,\infty]$, we denote by $L^p(\Omega)$ the Lebesgue space of (equivalence classes of) measurable functions $g$ for which the $p$-th power of the absolute value is Lebesgue integrable when $p<\infty$ (resp. essentially bounded when $p=\infty$). When the context is clear, we will write $\|g\|_p$ instead of $\|g\|_{L^p(\Omega)}$. Given $\alpha\in(0,1]$ and $p\in[1,\infty]$, $B_{p,\infty}^\alpha(\R^N)$ stands for the Nikol'skii space consisting in all measurable functions $g\in L^p(\R^N)$ such that
$$ [g]_{B_{p,\infty}^\alpha(\R^N)}:=\sup_{h\ne0}\frac{\|g(\cdot+h)-g\|_{L^p(\R^N)}}{|h|^\alpha}<+\infty. $$
We note that, when $p=\infty$, the space $B_{\infty,\infty}^\alpha(\R^N)$ coincides with the classical H\"older space~$C^{0,\alpha}(\R^N)$. For a set $E\subset\R^N$ and $g:E\to\R$, we set
$$[g]_{C^{0,\alpha}(E)}=\sup_{x\in E,\,y\in E,\,x\neq y}\frac{|g(x)-g(y)|}{|x-y|^{\alpha}}.$$

Let us finally recall some useful notions of regularity of a compact set $K$.

\begin{definition}
Let $\alpha\in(0,1]$. We say that a compact set $K\subset\R^N$ has $C^{0,\alpha}$ boundary if there exist $r>0$, $p\in\N$, $p$ rotations $(R_i)_{1\le i\le p}$ of $\R^N$, $p$ points $(z_i)_{1\le i\le p}$ of $\partial K$ and $p$ functions $(\psi_i)_{1\le i\le p}$ defined in the $(N-1)$-dimensional ball $B_r^{N-1}=\big\{x'\in\R^{N-1};\,|x'|<r\big\}$ of class $C^{0,\alpha}(B^{N-1}_r)$ and such that
\be\label{partialK1}
\partial K=\bigcup_{1\le i\le p}R_i\Big(\big\{x_N=\psi_i(x');\,x'\in B^{N-1}_r\big\}\Big)
\ee
and
\be\label{partialK2}
\overset{\circ}{K}\cap B_r(z_i)=R_i\Big(\big\{x_N>\psi_i(x');\,x'\in B^{N-1}_r\big\}\Big)\cap B_r(z_i)
\ee
for every $1\le i\le p$.
\end{definition}

\begin{definition}\label{DEF:KEPS}
Let $\alpha\in(0,1]$, let $K\subset\R^N$ be a compact convex set with non-empty interior ($\partial K$ is then automatically of class $C^{0,\alpha}$) and let~$(K_\eps)_{0<\eps\leq1}\subset\R^N$ be a family of compact, simply connected sets having $C^{0,\alpha}$ boundary. We say that $(K_\eps)_{0<\eps\leq1}$ is a \emph{family of} $C^{0,\alpha}$ \emph{deformations of} $K$ if the following conditions are fulfilled:
\begin{enumerate}
\item[(i)] $K\subset K_{\eps_1}\subset K_{\eps_2}$ for all $0<\eps_1\leq\eps_2\leq1$;
\item[(ii)] $K_\eps \to K$ as $\eps\downarrow0$ in $C^{0,\alpha}$, in the sense that there exist $r>0$, $p\in\N$, $p$ rotations $(R_i)_{1\le i\le p}$ of $\R^N$, $p$ points $(z_i)_{1\le i\le p}$ of $\partial K$, $p$ functions $(\psi_i)_{1\le i\le p}$ and $p$ families of functions $(\psi_{i,\eps})_{1\le i\le p,\,0<\eps\le 1}$ of class $C^{0,\alpha}(B^{N-1}_r)$ describing $\partial K$ and $\partial K_\eps$ as in~\eqref{partialK1} and~\eqref{partialK2} above, and such that
$$ \|\psi_i-\psi_{i,\eps}\|_{C^{0,\alpha}(B_r^{N-1})}\to0 \quad{\mbox{as }}\eps\downarrow0,\ \hbox{ for every }1\le i\le p.$$
\end{enumerate}
\end{definition}


\SE{Main results}\label{MAIN RES}

The Liouville property for the local problem~\eqref{EQ:BHM0} says that $u=1$ in $\overline{\R^N\setminus K}$ under some geometric conditions on $K$, in particular when $K$ is convex, see~\cite{BHM}. When the obstacle $K$ is convex, we prove that this Liouville property still holds for~\eqref{EQ:BHM1} with the nonlocal operator~$L$. We will actually prove several results, which correspond to various assumptions on the solutions $u$ and the data $f$ and $J$. We will also show the robustness of the Liouville type property with respect to small deformations of the obstacle $K$. The assumptions~\eqref{C1},~\eqref{C2} and~\eqref{C3} will be common assumptions of almost all results. In some statements, assumption~\eqref{C3} is sometimes replaced by the stronger assumption~\eqref{C4}.


\subsection{A first rough Liouville type result}

Under rather mild additional assumptions on $K$, we first state a ``rough" Liouville type property for the solutions of~\eqref{EQ:BHM1}, if $f$ is assumed to be non-negative on the range of $u$.

\begin{prop}\label{LargeGamma}
Let~$K\subset\R^N$ be a compact set such that $\R^N\setminus K$ is connected. Assume that $f\in C^1([0,1])$ and $J$ satisfies~\eqref{C2}. Let $\theta\in[0,1)$ and assume that $f\ge0$ in $[\theta,1]$. Let $u:\overline{\R^N\setminus K}\to[\theta,1]$ be a continuous solution of
\begin{align}
\left\{
\begin{array}{rl}
Lu+ f(u)=0  & \text{ in }\,\overline{\R^N\setminus K},\vspace{3pt} \\
u(x)\to1  & \text{ as }\,|x|\to+\infty.
\end{array}
\right.  \label{LIM:u0}
\end{align}
Then, $u=1$ in $\overline{\R^N\setminus K}$.
\end{prop}

One of the main goals of the paper is to understand under which conditions on $K$ this Liouville type property still holds or does not hold when $u$ ranges in the whole interval $[0,1]$.


\subsection{Liouville type properties for convex obstacles}

Our first main theorem is the following result dealing with continuous super-solutions to $L(u)+f(u)=0$ ranging in $[0,1]$.

\begin{theorem}\label{TH:LIOUVILLE}
Let~~$K\subset\R^N$ be a compact convex set. Assume~\eqref{C1},~\eqref{C2},~\eqref{C3} and let
\begin{align}
u\in C(\overline{\R^N\setminus K},[0,1]) \label{continue}
\end{align}
be a function satisfying
\begin{align}
\left\{
\begin{array}{rl}
Lu+f(u)\leq 0  & \text{ in }\,\overline{\R^N\setminus K}, \vspace{3pt}\\
u(x)\to1  & \text{ as }\,|x|\to+\infty.
\end{array}
\right.  \label{LIM:u}
\end{align}
Then, $u=1$ in $\overline{\R^N\setminus K}$.
\end{theorem}

If we ask for a solution of~\eqref{EQ:BHM1} instead of a super-solution, it turns out that the regularity or limiting conditions required on $u$ to obtain a Liouville type result can be considerably weakened, by strengthening the assumptions made on $f$ and/or $J$.

Firstly, the continuity assumption~\eqref{continue} can be relaxed provided the nonlinearity does not vary too much.  More precisely, we will prove the following result.

\begin{theorem}\label{PROP:CONTINUE}
Let $K\subset\R^N$ be a compact convex set. Assume~\eqref{C1},~\eqref{C2},~\eqref{C3} and
\begin{align}
\max_{[0,1]}f'<\frac{1}{2}. \label{CdN-f}
\end{align}
Let $u:\R^N\setminus K\to[0,1]$ be a measurable function satisfying
$$\left\{
\begin{array}{rl}
Lu+f(u)= 0  & \text{ a.e. in }\,\R^N\setminus K, \vspace{3pt}\\
u(x)\to1  & \text{ as }\,|x|\to+\infty.
\end{array}
\right.$$
Then, $u=1$ a.e. in $\R^N\setminus K$.
\end{theorem}

Secondly, assuming that $f$ and $J$ satisfy~\eqref{C4} instead of~\eqref{C3}, that $J\in L^2(\R^N)$ and is compactly supported, and that $f$ does not vary too much or $u$ is a priori uniformly continuous, then the assumptions on the asymptotic behaviour of $u$ at infinity can be noticeably weakened. More precisely, the following result holds.

\begin{theorem}\label{TH:LIOUVILLE2}
Let~$K\subset\R^N$ be a compact convex set and assume that $f$ and $J$ satisfy~\eqref{C1},~\eqref{C2} and~\eqref{C4}. Assume further that $J$ is compactly supported and $J\in L^2(\R^N)$. If $u:\overline{\R^N\setminus K}\to[0,1]$ is uniformly continuous in $\overline{\R^N\setminus K}$ and obeys
\be\left\{\begin{array}{rcl}
Lu+f(u) & \!\!=\!\! & 0\ \hbox{ in }\,\overline{\R^N\setminus K},\vspace{3pt}\\
\ds{\sup_{\R^N\setminus K}}\,u &  \!\!=\!\! & 1,\end{array}\right.  \label{LIM:cu9}
\ee
then $u=1$  in $\overline{\R^N\setminus K}$. Similarly, if~\eqref{CdN-f} holds and if $u:\R^N\setminus K\to[0,1]$ is a measurable function satisfying
\be\left\{\begin{array}{rcl}
Lu+f(u) & \!\!=\!\! & 0\ \hbox{ a.e. in }\,\R^N\setminus K,\vspace{3pt}\\
\ds\mathop{{\rm{ess}}\,{\rm{sup}}}_{\R^N\setminus K}\,u &  \!\!=\!\! & 1,\end{array}\right.  \label{LIM:u9}
\ee
then $u=1$ a.e. in $\R^N\setminus K$.
\end{theorem}

\begin{remark}
Condition~\eqref{CdN-f} ensures that $u$ actually has a uniformly continuous representative in $\overline{\R^N\setminus K}$ (see Lemma~\ref{LEMMA:INT} and Remark~\ref{RK:CVX}). However, if $u$ is already known to be uniformly continuous, then Theorem~\ref{TH:LIOUVILLE2} provides the same conclusion without assumption~\eqref{CdN-f} (see Lemma~\ref{lem:lim} for further details).
\end{remark}


\subsection{Robustness of the Liouville property for nearly convex obstacles $K$}

Under some flatness assumptions on $f$, and following a line of ideas in~\cite{Bouhours}, it turns out that the Liouville property is still available under small H\"older perturbations of a given convex obstacle $K$. Namely, the following result holds.

\begin{theorem}\label{TH:PERTURB}
Let $\alpha\in(0,1]$, let $K\subset\R^N$ be a compact convex set with non-empty interior and let $(K_\eps)_{0<\eps\leq1}$ be a family of $C^{0,\alpha}$ deformations of $K$. Assume~\eqref{C1},~\eqref{C2},~\eqref{C4} and suppose that $J\in B_{1,\infty}^\alpha(\R^N)$ and
$$\max_{[0,1]}f'<\inf_{0<\eps\leq 1}\inf_{x\in\R^N\setminus K_\eps}\,\|J(x-\cdot)\|_{L^1(\R^N\setminus K_\eps)}.$$
For $0<\eps\leq1$, let $L_\eps$ be the operator given by, for every $v\in L^{\infty}(\R^N\setminus K_\eps)$,
$$L_\eps v(x):=\int_{\R^N\setminus K_\eps} J(x-y)\big(v(y)-v(x)\big)\,\mathrm{d}y.$$
Then there exists $\eps_0\in(0,1]$ such that for all $\eps\in(0,\eps_0]$ the unique measurable solution $u_\eps$ of
\begin{equation}
\left\{
\begin{array}{rl}
L_\eps u_\eps+f(u_\eps)=0  & \text{a.e. in }\,\R^N\setminus K_\eps, \vspace{3pt}\\
0\leq u_\eps\leq 1 & \text{a.e. in }\,\R^N\setminus K_\eps, \vspace{3pt}\\
u_\eps(x)\to1  & \text{as }\,|x|\to+\infty,
\end{array}
\right.\label{eq:ueps}
\end{equation}
is $u_\eps=1$ a.e. in $\R^N\setminus K_\eps$.
\end{theorem}

\begin{remark}
It should be noted that the monotonicity assumption (i) in Definition~\ref{DEF:KEPS} has been made for simplicity and is not necessary for our purposes. Moreover, the conclusion of Theorem~\ref{TH:PERTURB} remains true whenever $(K_\eps)_{0<\eps\leq1}$ is a family of $C^{0,\alpha}$ deformations of any compact set~$K$ for which the conclusion of Theorem~\ref{PROP:CONTINUE} is valid. Since there exist some smooth, compact, non-convex and simply connected sets which are $C^{0,\alpha}$ close to a smooth, compact and convex set, Theorem~\ref{TH:PERTURB} implies that the Liouville property holds for some smooth, compact and non-convex obstacles, and then also for their $C^{0,\alpha}$ perturbations. Finally, we conjecture that the Liouville property of Theorem~\ref{PROP:CONTINUE} holds for any starshaped compact obstacle as well.\par
However, as in the local case (see~\cite[Theorem~6.5]{BHM}), the above Liouville type properties cannot be expected for general obstacles. For example, one can easily find counterexamples if $K$ is no longer simply connected. Take for instance $K=\overline{\mathcal{A}(1,2)}=\overline{B_2}\setminus B_1$ and suppose that $J$ is supported in $B_{1/2}$. Then, the function $u$ defined by
$$u(x)=\left\{
\begin{array}{ll}
1  & \text{ if }x\in\R^N\setminus B_2, \vspace{3pt}\\
0  & \text{ if }x\in\overline{B_1},
\end{array}
\right.$$
is a continuous solution of~\eqref{LIM:u0}. Yet, $u$ is not identically $1$ in the whole set $\R^N\setminus K$.
\end{remark}

\noindent{\bf{Outline of the paper.}} The following first sections are concerned with general results on the solutions to problems~\eqref{EQt} or~\eqref{EQ:BHM1}. Namely, in Section~\ref{3}, we show that the solutions are uniformly continuous, more precisely they have a uniformly continuous representative, if rather mild assumptions are made on~$f$. In Section~\ref{4}, we give several comparison principles that fit our purposes. We then use these comparison principles in Section~\ref{5} to construct a radially symmetric lower bound for the solutions. In Section~\ref{6}, we study an auxiliary problem which will enable us to pave the way towards the proof of Theorem~\ref{TH:LIOUVILLE2}. The remaining part of the paper is devoted to the proofs of our main results. In Section~\ref{8}, we prove, at a stroke, Theorems~\ref{TH:LIOUVILLE} and~\ref{PROP:CONTINUE}, and with more work we show how to relax the assumptions on~$u$ when the kernel~$J$ is compactly supported, that is we prove Theorem~\ref{TH:LIOUVILLE2}. In Section~\ref{9}, as a preliminary result we prove the rough Liouville-type result Proposition~\ref{LargeGamma} and then we establish our robustness result Theorem~\ref{TH:PERTURB}.


\SE{Some auxiliary regularity results}\label{3}

Throughout this section, $K$ is any compact subset of $\R^N$, $f$ is any $C^1(\R)$ function, and $J$ is any $L^1(\R^N)$ non-negative and radially symmetric kernel with unit mass. For $x\in\R^N$, we write
$$\mathcal{J}(x):=\int_{\R^N\setminus K}J(x-y)\,\mathrm{d}y.$$
Notice that $\mathcal{J}$ is a uniformly continuous function in $\R^N$. In the sequel, for any $\delta>0$, we will denote $K_{\delta}$ the closed thickening of $K$ with width $\delta$, defined by
$$ K_\delta:=K+\overline{B_\delta}.$$

We now prove that, when $J$ is compactly supported and $f'$ is not too large in $[0,1]$, then the measurable solutions $u$ to~\eqref{EQ:BHM1} are continuous far away from the obstacle.

\begin{lemma}\label{PROP:CONE}
Suppose that $K\subset\R^N$ is a compact set and that $J$ is supported in the ball $B_\delta$ for some $\delta>0$. Suppose that
\begin{align}
\max_{[0,1]}f'<1. \label{monotonie}
\end{align}
Let $u\in L^\infty(\R^N\setminus K,[0,1])$ be a solution of $Lu+f(u)=0$ a.e. in $\R^N\setminus K$. Then $u$ is uniformly continuous in $\overline{\R^N\setminus K_\delta}$, in the sense that $u$ has a representative in its class of equivalence that is uniformly continuous in $\overline{\R^N\setminus K_\delta}$. If, in addition, $J\in B_{1,\infty}^\alpha(\R^N)$ for some $\alpha\in(0,1]$, then $u\in C^{0,\alpha}(\overline{\R^N\setminus K_\delta})$ and $\big(1-\max_{[0,1]}f'\big)[u]_{C^{0,\alpha}(\overline{\R^N\setminus K_\delta})}\leq [J]_{B_{1,\infty}^\alpha(\R^N)}$.
\end{lemma}

\begin{proof}
For every $x$ and $y$ in $\R^N\setminus K_{\delta}$, we have
$$\baa{rcl}
Lu(x)-Lu(y) & = & \ds\int_{\R^N\setminus K}u(z)\big(J(x-z)-J(y-z)\big)\mathrm{d}z\vspace{3pt}\\
& & \ds-u(x)\int_{\R^N\setminus K}J(x-z)\,\mathrm{d}z+u(y)\int_{\R^N\setminus K}J(y-z)\,\mathrm{d}z \vspace{3pt}\\
& = & \ds-\mathcal{J}(x)\,u(x)+\mathcal{J}(y)\,u(y)+\int_{\R^N\setminus K}u(z)\big(J(x-z)-J(y-z)\big)\mathrm{d}z.\eaa$$
Since $J$ has unit mass and is supported in $B_\delta$, we get that $\mathcal{J}(x)=\mathcal{J}(y)=1$. Therefore,
$$Lu(x)-Lu(y)+u(x)-u(y)=\int_{\R^N\setminus K}u(z)\big(J(x-z)-J(y-z)\big)\mathrm{d}z.$$\par
Now, remember that $u$ is a solution to $Lu+f(u)=0$ a.e. in $\R^N\setminus K$. In particular,
there exists a measurable negligible set $E$ such that $Lu(z)+f(u(z))=0$ (and $u(z)\in[0,1]$) for all $z\in\R^N\setminus(K\cup E)$. Hence, letting
$$g(t):=t-f(t)$$
for $t\in[0,1]$, we obtain that
\be\label{gxy}
\forall\,x,y\in\R^N\setminus(K_\delta\cup E),\ \ g(u(x))-g(u(y))\!=\!\int_{\R^N\setminus K}\!\!u(z)\big(J(x\!-\!z)\!-\!J(y\!-\!z)\big)\mathrm{d}z=:h(x,y).
\ee
Notice that, since $J\in L^1(\R^N)$ and $u\in L^{\infty}(\R^N\setminus K)$, the function $h$ defined by the right-hand side of the previous equation can actually be defined in $\R^N\times\R^N$ and it is uniformly continuous in $\R^N\times\R^N$. Furthermore, by~\eqref{monotonie}, the function $g\in C^1([0,1])$ is such that $g'>0$ in $[0,1]$. It is then a $C^1$ diffeomorphism from $[0,1]$ to $[g(0),g(1)]=[-f(0),1-f(1)]$. Let us denote $g^{-1}:[g(0),g(1)]\to[0,1]$ its reciprocal.\par
Fix $y_0\in\R^N\setminus(K_\delta\cup E)$. For every $x\in\R^N\setminus(K_\delta\cup E)$,~\eqref{gxy} yields
$$g(u(y_0))+h(x,y_0)=g(u(x))\in[g(0),g(1)].$$
Since the function $x\mapsto g(u(y_0))+h(x,y_0)$ is continuous (in the whole $\R^N$) and since $E$ is negligible, it follows that $g(u(y_0))+h(x,y_0)\in[g(0),g(1)]$ for all $x\in\R^N\setminus K_\delta$ (since any point of the open set $\R^N\setminus K_\delta$ is the limit of a sequence of points in $\R^N\setminus(K_\delta\cup E)$). Define now
$$\tilde{u}(x)=g^{-1}\big(g(u(y_0))+h(x,y_0)\big)\ \ \hbox{ for }x\in\overline{\R^N\setminus K_\delta}.$$
By~\eqref{gxy}, one has $\tilde{u}=u$ in $\R^N\setminus(K_\delta\cup E)$. Furthermore, $\tilde{u}$ is uniformly continuous in $\overline{\R^N\setminus K_\delta}$ owing to its definition, since $h$ is uniformly continuous in $\R^N\times\R^N$ and $g^{-1}$ is $C^1$ hence Lipschitz continuous in $[g(0),g(1)]$.\par
Even if it means redefining $u$ by $\tilde{u}$ in $\overline{\R^N\setminus K_\delta}$, it follows that $u$ is uniformly continuous in $\overline{\R^N\setminus K_\delta}$ and that~\eqref{gxy} holds, by continuity, for all $x,y\in\overline{\R^N\setminus K_\delta}$. In particular, since $0\le u\le 1$ in $\R^N\setminus K$, we get that
\be\label{L1control}
\forall\,x,y\in\overline{\R^N\setminus K_\delta},\quad|g(u(x))-g(u(y))|\le\|J(\cdot+x-y)-J\|_{L^1(\R^N)}.
\ee\par
\noindent Finally, if $J\in B_{1,\infty}^\alpha(\R^N)$ with $\alpha\in(0,1]$, then~\eqref{L1control} yields $g(u)\in C^{0,\alpha}(\overline{\R^N\setminus K_\delta})$ and
$$[g(u)]_{C^{0,\alpha}(\overline{\R^N\setminus K_\delta})}\leq \sup_{h\ne0}\,\frac{\|J(\cdot+h)-J\|_{L^1(\R^N)}}{|h|^\alpha}= [J]_{B_{1,\infty}^\alpha(\R^N)}.$$
Since $\max_{[g(0),g(1)]}|(g^{-1})'|\le(1-\max_{[0,1]}f')^{-1}$ and $0\le u\le 1$, one concludes that $u\in C^{0,\alpha}(\overline{\R^N\setminus K_\delta})$ and $\big(1-\max_{[0,1]}f'\big)[u]_{C^{0,\alpha}(\overline{\R^N\setminus K_\delta})}\leq [J]_{B_{1,\infty}^\alpha(\R^N)}$.
\end{proof}

We now establish a regularity result for $u$ in the whole set $\overline{\R^N\setminus K}$ for flatter nonlinearities.

\begin{lemma}\label{LEMMA:INT}
Suppose that $K\subset\R^N$ is a compact set and that
\begin{align}
\max_{[0,1]}f'<\inf_{\R^N\setminus K}\,\mathcal{J}.   \label{monotonie2}
\end{align}
Let $u\in L^\infty(\R^N\setminus K,[0,1])$ be a solution of $Lu+f(u)=0$ a.e. in $\R^N\setminus K$. Then, $u$ can be redefined up to a negligible set and extended as a uniformly continuous function in $\overline{\R^N\setminus K}$. If, in addition, $J\in B_{1,\infty}^\alpha(\R^N)$ for some $\alpha\in(0,1]$, then $u\in C^{0,\alpha}(\overline{\R^N\setminus K})$ and
$$\Big(\inf_{\R^N\setminus K}\mathcal{J}-\max_{[0,1]}f'\Big)\,[u]_{C^{0,\alpha}(\overline{\R^N\setminus K})}\leq 2\,[J]_{B_{1,\infty}^\alpha(\R^N)}.$$
\end{lemma}

\begin{remark}\label{RK:CVX}
In the case of a compact convex obstacle $K$, then the conclusion of Lemma~\ref{LEMMA:INT} still holds if~\eqref{monotonie2} is replaced by
$$\max_{[0,1]}f'<\frac{1}{2}.$$
Indeed, $\mathcal{J}\geq1/2$ in $\R^N\setminus K$ when $K$ is convex (remember also that $J$ is always assumed to be a non-negative radially symmetric kernel with unit mass). The bound $1/2$ is somehow optimal, since $K$ can be as large as desired, still in the class of compact convex sets. However, this bound deteriorates considerably if $K$ is only starshaped, as the infimum of $\mathcal{J}$ in $\R^N\setminus K$ can become arbitrarily small. Roughly speaking, the less convex the obstacle $K$, the flatter the nonlinearity $f$ needs to be to insure~\eqref{monotonie2} and the interior continuity of the solution $u$.
\end{remark}

\begin{proof}[Proof of Lemma~$\ref{LEMMA:INT}$]
Reasoning exactly as in the proof of Lemma~\ref{PROP:CONE}, there exists a measurable negligible set $E$ such that
\be\label{gxy2}
\forall\,x,y\in\R^N\setminus(K\cup E),\ \ G(x,u(x))-G(y,u(y))=h(x,y),
\ee
where $h(x,y)$ is defined in~\eqref{gxy} (remember also that $h$ is uniformly continuous in $\R^N\times\R^N$) and
$$G(x,s)=\mathcal{J}(x)\,s-f(s)\ \hbox{ for }(x,s)\in\R^N\times[0,1].$$
By~\eqref{monotonie2} and the continuity of $\mathcal{J}$, the function $G$ is such that $\partial_sG(x,s)>0$ for all $(x,s)\in\overline{\R^N\setminus K}\times[0,1]$. For every $x\in\overline{\R^N\setminus K}$, the function $G(x,\cdot)$ is then a $C^1$ diffeomorphism from~$[0,1]$ to $[G(x,0),G(x,1)]$. Let us denote $H_x:[G(x,0),G(x,1)]\to[0,1]$ its reciprocal, that is, $H_x(G(x,t))=t$ for all $x\in\overline{\R^N\setminus K}$ and $t\in[0,1]$.\par
Fix $y_0\in\R^N\setminus(K\cup E)$. For every $x\in\R^N\setminus(K\cup E)$,~\eqref{gxy2} yields
$$G(y_0,u(y_0))+h(x,y_0)=G(x,u(x))\in[G(x,0),G(x,1)].$$
Since the function $x\mapsto G(y_0,u(y_0))+h(x,y_0)$ is continuous (in the whole space $\R^N$), since $G$ is itself continuous in $\R^N\times[0,1]$ and since $E$ is negligible, it follows that $G(y_0,u(y_0))+h(x,y_0)\in[G(x,0),G(x,1)]$ for all $x$ in the open set $\R^N\setminus K$. Define now
$$\tilde{u}(x)=H_x\big(G(y_0,u(y_0))+h(x,y_0)\big)\ \ \hbox{ for }x\in\overline{\R^N\setminus K}.$$
By~\eqref{gxy2}, one has $\tilde{u}=u$ in $\R^N\setminus(K\cup E)$. Furthermore, $\tilde{u}$ is continuous in $\overline{\R^N\setminus K}$ owing to its definition, since $h$ is continuous in $\R^N\times\R^N$ and $(x,s)\mapsto H_x(s)$ is continuous in the set $\big\{(x,s)\in\overline{\R^N\setminus K}\times\R;\,s\in[G(x,0),G(x,1)]\big\}$. Even if it means redefining $u$ by $\tilde{u}$ in $\R^N\setminus K$ and extending it by $\tilde{u}$ in $\overline{\R^N\setminus K}$, it follows that $u$ is continuous in $\overline{\R^N\setminus K}$ and that~\eqref{gxy2} holds, by continuity, for all $x$, $y$ in the open set $\R^N\setminus K$ and then in $\overline{\R^N\setminus K}$. In particular, since $0\le u\le 1$ in $\overline{\R^N\setminus K}$, we get that
\be\label{Guxy}
\forall\,x,y\in\overline{\R^N\setminus K},\quad|G(x,u(x)))-G(y,u(y))|\leq\|J(\cdot+x-y)-J\|_{L^1(\R^N)}.
\ee\par
\noindent Finally, define
$$\beta:=\inf_{\R^N\setminus K}\mathcal{J}-\max_{[0,1]}f'>0,$$
the positivity of $\beta$ resulting from~\eqref{monotonie2}. {F}rom~\eqref{Guxy} together with the definition of $G$ and the inequalities $0\le u\le 1$ in $\overline{\R^N\setminus K}$, one infers that, for all $x,y\in\overline{\R^N\setminus K}$,
$$\baa{rcl}
\big|\mathcal{J}(x)\,(u(x)\!-\!u(y))-(f(u(x))\!-\!f(u(y)))\big| & \!\!\!\le\!\!\! & \|J(\cdot\!+\!x\!-\!y)-J\|_{L^1(\R^N)}+\big|u(y)\,(\mathcal{J}(x)\!-\!\mathcal{J}(y))\big|\vspace{3pt}\\
& \!\!\!\le\!\!\! & 2\,\|J(\cdot+x-y)-J\|_{L^1(\R^N)}.\eaa$$
It follows from the mean value theorem and the above definition of $\beta$ that
$$\beta\,|u(x)-u(y)|\le2\,\|J(\cdot+x-y)-J\|_{L^1(\R^N)}$$
for all $x,y\in\overline{\R^N\setminus K}$. In particular, the function $u$ is uniformly continuous in $\overline{\R^N\setminus K}$. Furthermore, if $J\in B^\alpha_{1,\infty}(\R^N)$ for some $\alpha\in(0,1]$, then $u\in C^{0,\alpha}(\overline{\R^N\setminus K})$ and $\beta\,[u]_{C^{0,\alpha}(\overline{\R^N\setminus K})}\leq 2\,[J]_{B_{1,\infty}^\alpha(\R^N)}$. The proof of Lemma~\ref{LEMMA:INT} is thereby complete.
\end{proof}


\SE{Comparison principles}\label{4}

In this section, we collect some comparison principles that fit for our purposes. Throughout this section, $K$ is any compact subset of $\R^N$, $f$ is any $C^1(\R)$ function, and $J$ is any $L^1(\R^N)$ non-negative and radially symmetric kernel with unit mass.

We start with a weak maximum principle.

\begin{lemma}[Weak maximum principle]\label{WEAK}
Assume that
\begin{equation}\label{well+2}
{\mbox{$f'\le -c_1$ in $[1-c_0,+\infty),\ $ for some $c_0>0$, $c_1>0$.}}
\end{equation}
Let~$H\subset\R^N$ be an open affine half-space such that $K\subset H^c=\R^N\setminus H$. Let $u,v\in L^{\infty}(\R^N\setminus K)$ be such that
\begin{equation}\label{H:CO}
u, \;v\in C\big( \overline{H} \big)
\end{equation}
and
\begin{equation}\label{H:EQ}
\left\{\baa{ll}
Lu+f(u)\leq 0 & {\mbox{ in }}\overline{H},\vspace{3pt}\\
Lv+f(v)\ge 0 & {\mbox{ in }}\overline{H}.\eaa\right.
\end{equation}
Assume also that
\begin{equation}\label{up:c0}
u\ge 1-c_0 \quad{\mbox{ in }}\overline{H},
\end{equation}
that
\begin{equation}\label{H:lim}
\limsup_{|x|\to+\infty} \big(v(x)-u(x)\big)\le 0
\end{equation}
and that
\begin{equation}\label{H:SI}
v\le u\quad{\mbox{ a.e. in }} H^c\setminus K.
\end{equation}
Then,~$v\le u$ a.e. in~$\R^N\setminus K$.
\end{lemma}

\begin{proof} We let~$w:=v-u$. We want to prove that $w\le0$ a.e. in $\R^N\setminus K$. {F}rom~\eqref{H:SI}, we only have to show that $w\leq 0$ in $\overline{H}$ (remember that from~\eqref{H:CO} both functions $u$ and $v$ are assumed to be continuous in $\overline{H}$). We argue by contradiction and we suppose that $\sup_{\overline{H}}w>0$. Then, thanks to~\eqref{H:SI}, one has $\sup_{\R^N\setminus K} w=\sup_{\overline{H}} w>0$ and
there exists a sequence~$(x_j)_{j\in\N}$ in $\overline{H}$ such that
$$ \lim_{j\to+\infty} w(x_j)=\sup_{\R^N\setminus K} w>0.$$
It follows then from~\eqref{H:lim} that the sequence $(x_j)_{j\in\N}$ is bounded. Thus, up to extraction of a subsequence,
there exists a point $\bar x\in \overline{H}$ such that $x_j\to\bar{x}$ as $j\to+\infty$, hence $w(\bar{x})=\lim_{j\to+\infty}w(x_j)>0$ by~\eqref{H:CO}. As a consequence,~\eqref{H:EQ} yields
\begin{equation}\label{INS:H}
L w(\bar x)=Lv(\bar x) -Lu(\bar x)\ge-f\big(v(\bar x)\big)+f\big(u(\bar x)\big)=-w(\bar x)\int_0^1 f'\big( tv(\bar x)+(1-t)u(\bar x) \big)\,\D t.
\end{equation}
Moreover, combining~\eqref{up:c0} and $w(\bar{x})>0$, we obtain that $v(\bar x)=w(\bar x)+ u(\bar x)>u(\bar x)\ge 1-c_0$, and so $ tv(\bar x)+(1-t)u(\bar x) \ge1-c_0$ for all~$t\in [0,1]$. {F}rom this and~\eqref{well+2}, we conclude that $ f'\big( tv(\bar x)+(1-t)u(\bar x) \big)\le-c_1<0$ for all~$t\in [0,1]$. This inequality, together with $w(\bar{x})>0$, yields
$$-w(\bar x)\int_0^1 f'\big( tv(\bar x)+(1-t)u(\bar x) \big)\,\D t>0.$$
By inserting this information into~\eqref{INS:H}, we get $ L w(\bar x)>0$. That is, recalling~\eqref{DEF:L} and the nonnegativity of $J$,
$$0<Lw(\bar x)=\int_{\R^N\setminus K} J(\bar{x}-y)\,\big( w(y)-w(\bar x) \big)\,\D y=\int_{\R^N\setminus K} J(\bar{x}-y)\,\bigg( w(y)-\sup_{\R^N\setminus K} w \bigg)\,\D y \le 0.$$
This is a contradiction, and so the desired result is established.
\end{proof}

The next lemma is concerned with a strong maximum principle.

\begin{lemma}[Strong maximum principle]\label{STRONG:LE23}
Assume that $J$ satisfies~\eqref{C2}, with $0\le r_1<r_2$. Let $H\subset\R^N$ be an open affine half-space such that $K\subset H^c$. Let $u,v\in L^{\infty}(\R^N\setminus K)$ satisfy~\eqref{H:CO} and~\eqref{H:EQ}. Assume also that
\begin{equation}\label{3beta:2}
v\le u \quad{\mbox{ a.e. in }} \R^N\setminus K
\end{equation}
and that there exists~$\bar x\in\overline{H}$ such that $v(\bar x)=u(\bar x)$. Then,
$$v= u \quad\mbox{ a.e. in }(H+B_{r_2})\setminus K.$$
\end{lemma}

\begin{proof} We let~$w:=v-u$. Notice that $w(\bar x)=0$. As a consequence, using~\eqref{H:EQ}, we can write
$$ L w(\bar x)= Lv(\bar x) -Lu(\bar x)\ge-f(v(\bar x))+f(u(\bar x))=0. $$
On the other hand, $w(y)\le0=w(\bar x)$ for a.e. $y\in\R^N\setminus K$, thanks to~\eqref{3beta:2}, and therefore
$$Lw(\bar x)=\int_{\R^N\setminus K} J(\bar{x}-y)\,\big( w(y)-w(\bar x) \big)\,\D y\le0.$$
Hence, $Lw(\bar{x})=0$ and
$$0=\int_{\R^N\setminus K}J(\bar{x}-y)\,\big(w(y)-w(\bar{x})\big)\,\D y=\int_{\R^N\setminus K}J(\bar{x}-y)\,w(y)\,\D y.$$
{F}rom our assumptions, we have $w\le0$ a.e. in $\R^N\setminus K$. Accordingly, since $J$ is such that $J>0$ a.e. in the annulus $\mathcal{A}(r_1,r_2)$ from the general assumption~\eqref{C2}, it follows that
$$w(x)=0,\hbox{ i.e. }v(x)=u(x),\ \hbox{ for a.e. }x\in\mathcal{A}(\bar x,r_1,r_2)\cap\R^N\setminus K.$$
In particular, since $u$ and $v$ are continuous in $\overline{H}$ and $H\subset\R^N\setminus K$, we get that
$$v(x)=u(x)\ \hbox{ for all }x\in\overline{\mathcal{A}(\bar x,r_1,r_2)}\cap\overline{H}=:\Omega_1(\bar{x}).$$
Applying the same arguments as above to the new set of contact points $\Omega_1(\bar{x})$ we obtain that $v(x)=u(x)$ for all $x\in\overline{\mathcal{A}(x_1,r_1,r_2)}\cap\overline{H}$ and for all $x_1\in\Omega_1(\bar{x})$. As a consequence, $v(x)=u(x)$ for all $x\in\overline{B_{\mu}(\bar{x})}\cap\overline{H}$ with $\mu:=r_2-r_1$. Iterating this procedure over again implies that $v(x)=u(x)$ for each $x$ in $\overline{B_{2\mu}(\bar{x})}\cap\overline{H}$ and so on in $\overline{B_{k\mu}(\bar{x})}\cap\overline{H}$ for any $k\in\N$. Hence, $v=u$ in~$\overline{H}$.\par
Therefore, as in the beginning of the proof, it follows that $Lw(x)=0$ for all $x\in\overline{H}$ and
$$v=u\ \hbox{ a.e. in }\big(\overline{H}+\mathcal{A}(r_1,r_2)\big)\cap\big(\R^N\setminus K\big)=(H+B_{r_2})\setminus K.$$
The proof of Lemma~\ref{STRONG:LE23} is thereby complete.
\end{proof}

Finally, we derive a sweeping-type result in the spirit of Serrin's sweeping
theorem~\cite{Serrin} (see also~\cite{NABB},
and page~29 in~\cite{Pucci} for a very clear explanation of the method).

\begin{lemma}[Sweeping principle]\label{TH:SWEEP}
Assume that $J$ satisfies~\eqref{C2}, with $0\leq r_1<r_2$. Let $g:\R\to\R$ be a continuous function, let $a,b,s_1,s_2,s_3,s_4$ be some real numbers such that $a\le b$ and $r_2\le s_1\le s_2<s_3\le s_4$. Let $u\in C(\overline{\mathcal{A}(s_1,s_4)})$ satisfy
\begin{equation}\label{eq:u}
\int_{\mathcal{A}(s_1,s_4)}J(x-y)\,u(y)\,\D y -u(x)+g\big(u(x)\big)\le 0\ {\mbox{ for all }}x\in\overline{\mathcal{A}(s_1,s_4)},
\end{equation}
and
\begin{equation}\label{eq:u1}
\int_{\mathcal{A}(s_1,s_4)}J(x-y)\,u(y)\,\D y -u(x)+g\big(u(x)\big)< 0\ \hbox{ for all }x\in\mathcal{A}(s_2,s_3).
\end{equation}
Let $(w_\tau)_{\tau\in[a,b]}$ be a continuous family in $C(\overline{\mathcal{A}(s_1,s_4)})$ such that
\begin{equation}\label{eq:tau}
\int_{\mathcal{A}(s_1,s_4)}J(x-y)\,w_{\tau}(y)\,\D y -w_\tau(x)+g\big(w_\tau(x)\big)\ge 0\ {\mbox{ for all }}x\in\overline{\mathcal{A}(s_1,s_4)}.
\end{equation}
Assume further that there exists $\tau_0\in[a,b]$ such that $w_{\tau_0}\le u$ in $\overline{\mathcal{A}(s_1,s_4)}$. Then $w_\tau\le u$ in~$\overline{\mathcal{A}(s_1,s_4)}$ for every $\tau\in[a,b]$.
\end{lemma}

\begin{proof}
Let us define $\Sigma \subset [a,b]$ to be the following set:
$$\Sigma:=\Big\{\tau \in [a,b];\ w_\tau\le u\hbox{ in }\overline{\mathcal{A}(s_1,s_4)} \Big\}.$$
To prove the theorem, we will show that $\Sigma$ is a non-empty open and closed set relatively to~$[a,b]$. It will then follow that $\Sigma=[a,b]$ and the theorem will be proved. First of all, by definition,  $\Sigma$ is a closed subset of $[a,b]$ and $\tau_0 \in \Sigma$. To finish our proof, it  remains to show that~$\Sigma$ is an open set relatively to $[a,b]$. So let us  pick $\tau \in \Sigma$. We have $w_\tau\le u$ in $\overline{\mathcal{A}(s_1,s_4)}$. By continuity of $u$ and $w_\tau$ in the compact set $\overline{\mathcal{A}(s_1,s_4)}$, either $\max_{\overline{\mathcal{A}(s_1,s_4)}}(w_\tau-u)<0$ or there exists $z\in \overline{\mathcal{A}(s_1,s_4)}$ such that  $w_\tau(z)=u(z)$. In the latter case, using $w_\tau\le u$ in $\overline{\mathcal{A}(s_1,s_4)}$ together with~\eqref{eq:u} and~\eqref{eq:tau} at the point~$z$, we get that
$$0\le  \int_{\mathcal{A}(s_1,s_4)}J(z-y)\big(u(y)-w_{\tau}(y)\big)\,\D y\le 0.$$
Using the continuity of both $u$ and $w_\tau$ and the fact that $J>0$ a.e. in $\mathcal{A}(r_1,r_2)$ for some $0\le r_1<r_2$ by~\eqref{C2}, it follows that $w_\tau=u$ in $\overline{\mathcal{A}(z,r_1,r_2)}\cap\overline{\mathcal{A}(s_1,s_4)}$ (which is nonempty since $r_2\leq s_1$) and then $w_\tau=u$ in $\overline{\mathcal{A}(z',r_1,r_2)}\cap\overline{\mathcal{A}(s_1,s_4)}$ for all $z'\in\overline{\mathcal{A}(z,r_1,r_2)}\cap\overline{\mathcal{A}(s_1,s_4)}$. In particular, it is easy to see that
there exists $r>0$ such that $w_\tau=u$ in $\overline{B_r(z)}\cap\overline{\mathcal{A}(s_1,s_4)}$. As a consequence, the non-empty set $\big\{x\in\overline{\mathcal{A}(s_1,s_4)};\ w_\tau(x)=u(x)\big\}$ is both (obviously) closed and open relatively to the (connected) set $\overline{\mathcal{A}(s_1,s_4)}$ and it is thus equal to $\overline{\mathcal{A}(s_1,s_4)}$. In other words, $w_\tau=u$ in $\overline{\mathcal{A}(s_1,s_4)}$, hence
$$\int_{\mathcal{A}(s_1,s_4)}J(x-y)\,u(y)\,\D y -u(x)+g\big(u(x)\big)=0  \quad{\mbox{ for all }}x\in\overline{\mathcal{A}(s_1,s_4)},$$
contradicting~\eqref{eq:u1} in $\mathcal{A}(s_2,s_3)$. Therefore, we must have $\max_{\overline{\mathcal{A}(s_1,s_4)}}(w_\tau-u)<0$. Since $w_\tau$ is continuous with respect to $\tau$ in the uniform norm,
there exists~$\delta>0$ such that $w_{\tau'}\le u$ in $\overline{\mathcal{A}(s_1,s_4)}$ for all $\tau'\in (\tau-\delta,\tau+\delta)\cap[a,b]$. Hence, $(\tau-\delta,\tau+\delta)\cap[a,b]\subset \Sigma$, which shows that $\Sigma$ is open relatively to $[a,b]$.
\end{proof}

\begin{remark} The previous arguments immediately show that, when $r_1=0$ in~\eqref{C2},
the sweeping principle holds in any compact connected set $F$. Namely, if $J$ satisfies~\eqref{C2} with $r_1=0$, if $u\in C(F)$ satisfies~\eqref{eq:u} with $F$ instead of $\overline{\mathcal{A}(s_1,s_4)}$ and the strict inequality somewhere in $F$, if $(w_\tau)_{\tau\in[a,b]}$ is a continuous family in $C(F)$ satisfying~\eqref{eq:tau} with $F$ instead of $\overline{\mathcal{A}(s_1,s_4)}$ and if $w_{\tau_0}\le u$ in $F$ for some $\tau_0\in[a,b]$, then $w_\tau\le u$ in $F$ for every $\tau\in[a,b]$.
\end{remark}


\SE{Construction of radially symmetric lower bounds}\label{5}

In this section, we derive a first lower bound on continuous non-negative super-solutions $u$ of~\eqref{LIM:u}  that we constantly use along this paper. Throughout this section, $K$ is any compact subset of $\R^N$, $f$ is any $C^1(\R)$ function, and $J$ is any $L^1(\R^N)$ non-negative and radially symmetric kernel with unit mass. We recall that $J_1$ is the non-negative even $L^1(\R)$ kernel with unit mass defined for a.e. $y_1\in\R$ by
$$ J_1(y_1):=\int_{\R^{N-1}} J(y_1,y_2,\cdots,y_N)\,\D y_2\,\cdots\,\D y_N,$$
and that assumption~\eqref{C3} means the existence of a continuous increasing function $\phi:\R\to(0,1)$ such that
\be\label{LI:PHI}\left\{\baa{l}
\displaystyle\int_{\R}J_1(\tau -\sigma)\,\big(\phi(\sigma)-\phi(\tau) \big)\,\D \sigma +f\big(\phi(\tau)\big) \ge 0\; \text{ for all }\; \tau \in \R,\vspace{3pt}\\
\phi(-\infty)= 0,\ \ \phi(+\infty)=1.\eaa\right.
\ee

Then, for such $\phi$, we establish the following lemma:

\begin{lemma}\label{sub-solution}
Assume that $f$ and $J$ satisfy~\eqref{well+2} and~\eqref{C3}, let $\gamma\in(0,1]$ and let $u\in C\big( \overline{ \R^N\setminus K },\;[\gamma,1]\big)$ be a function satisfying~\eqref{LIM:u}. Then, there exists  $r_0>0$ such that
$$ \phi(|x|-r_0)\leq u(x)\ \hbox{ for all }x\in\overline{\R^N\setminus K}.$$
\end{lemma}

\begin{proof}
Since $u(x)\to1$ as $|x|\to+\infty$, there exists~$R_0>0$ so large that ~$K\subset B_{R_0}$ and~$u\ge 1-c_0$ in~$\R^N\setminus B_{R_0}$, where $c_0>0$ is given in~\eqref{well+2}. By~\eqref{LI:PHI},
there exists~$A>0$ such that $\phi\le\gamma$ in $(-\infty,-A]$. Define
$$r_0=R_0+A>0$$
and let us check that the conclusion of Lemma~\ref{sub-solution} holds with this real number $r_0$.\par
Let $e$ be any unit vector of $\R^N$, that is, $e\in \partial B_1=\mathbb{S}^{N-1}$. For~$r\in\R$, let $\phi_{r,e}$ be the function defined by
$$ \phi_{r,e}(x):= \phi(e\cdot x-r)\ \hbox{ for }x\in\R^N,$$
where $e\cdot x$ stands for the standard inner product in~$\R^N$. Let~$(e_1,\cdots,e_N)$ be the canonical basis of~$\R^N$ and let~${\mathcal{R}}$ be a rotation such that~$e={\mathcal{R}}e_1$. Set now~$\tilde e_i:={\mathcal{R}} e_i$ for all~$i\in\{2,\cdots,N\}$ and, for $x$, $y\in\R^N$ and $r\in\R$, let us define $x^*=x-re$, $y^*=y-re$ and
$$\left\{\baa{lclclcl}
X & = & (X_1,\cdots,X_N) & = & (x^*\cdot e, x^*\cdot\tilde e_2,\cdots,x^*\cdot \tilde e_N) & = & {\mathcal{R}}^{-1} y^*,\vspace{3pt}\\
Y & = & (Y_1,\cdots,Y_N) & = & (y^*\cdot e, y^*\cdot\tilde e_2,\cdots,y^*\cdot \tilde e_N) & = & {\mathcal{R}}^{-1} y^*.\eaa\right.$$
Since $J$ is rotationally invariant, we deduce from~\eqref{LI:PHI} that, for all $x\in\R^N$ and $r\in\R$,
\be\label{EQ:phi}\baa{rcl}
\displaystyle L_{\R^N}\phi_{r,e}(x):=\int_{\R^N}\!\!\!J(x\!-\!y)\,\big( \phi_{r,e}(y)\!-\!\phi_{r,e}(x) \big)\,\D y
& \!\!=\!\! & \displaystyle\int_{\R}\! J_1(X_1\!-\!Y_1)\,\big( \phi(Y_1)\!-\!\phi(X_1) \big)\,\D Y_1\vspace{3pt}\\
& \!\!\geq\!\! & -f\big(\phi(X_1)\big)\vspace{3pt}\\
& \!\!=\!\! & -f\big(\phi(x\cdot e-r)\big)= -f\big(\phi_{r,e}(x)\big).\eaa
\ee
Set~$H_e:= \{ x\in\R^N;\ x\cdot e> R_0\}$ (notice that $\overline{H_e}\cap K=\emptyset$). We remark that, if~$r\ge r_0$ and~$x\in H_e^c\setminus K$, then
$$\phi_{r,e}(x)=\phi(x\cdot e-r) \le\phi(R_0-r_0)=\phi(-A)\le\gamma\le u(x).$$
Furthermore, if~$r\ge r_0$, $y\in K$ and $x\in\overline{H_e}$, then~$y\cdot e-r\le |y|-r\le R_0-r$ and
$$\phi_{r,e}(x)=\phi(x\cdot e-r)\ge \phi(R_0-r)\ge\phi(y\cdot e-r) =\phi_{r,e}(y).$$
Accordingly, by~\eqref{EQ:phi} and the definition of $H_e$, for any $r\ge r_0$ and~$x\in\overline{H_e}$,
\be\label{eq15}
L \phi_{r,e}(x)=L_{\R^N} \phi_{r,e}(x)-\int_{K} J(x-y)\,\big( \phi_{r,e}(y)-\phi_{r,e}(x) \big)\,\D y\ge -f\big(\phi_{r,e}(x)\big).
\ee
Consequently, we can exploit the weak comparison principle of Lemma~\ref{WEAK} (used here with~$H=H_e\subset\R^N\setminus K$ and~$v=\phi_{r_0,e}$) and deduce that
$$ \phi(x\cdot e-r_0)=\phi_{r_0,e}(x)\le u(x)$$
for every~$x\in\R^N\setminus K$ and also for every $x\in\overline{\R^N\setminus K}$ by continuity. This inequality holds for every $e\in\partial B_1$, while $r_0>0$ does not depend on $e$. In particular, taking into account the possible choice of~$e=x/|x|$ if $x\neq0$ (and any $e\in\partial B_1$ if $x=0$), we conclude that
$$ \phi(|x|-r_0) \le u(x) \quad{\mbox{ for all }}x\in\overline{\R^N\setminus K}.$$
This proves Lemma~\ref{sub-solution}.
\end{proof}

\begin{remark}\label{remgamma}
If $\R^N\setminus K$ is connected, if $f(0)\ge0$ and if $J$ satisfies~\eqref{C2} with $r_1=0$ (for instance, if $J$ is continuous at the origin with $J(0)>0$), then Lemma~\ref{sub-solution} holds with $\gamma=0$. Indeed, these additional assumptions imply that $\inf_{\R^N\setminus K}u>0$. If not, then by continuity of $u$ and the limiting conditions in~\eqref{LIM:u},
there exists~$x_0\in\overline{\R^N\setminus K}$ such that $u(x_0)=0$. Thus, by~\eqref{LIM:u} and $f(0)\ge0$,
$$0\ge Lu(x_0)=\int_{\R^N\setminus K}J(x_0-y)\,\big(u(y)-u(x_0)\big)\,\D y$$
and $u(y)=u(x_0)=0$ for all $y\in\overline{B_{r_2}(x_0)}\cap\overline{\R^N\setminus K}$. Therefore, $u(y)=0$ for all $y\in\overline{\R^N\setminus K}$ by repeating this argument and by connectedness of $\R^N\setminus K$. This contradicts the limit $u(y)\to1$ as $|y|\to+\infty$. Finally, $\inf_{\R^N\setminus K}u>0$ and the conclusion of Lemma~\ref{sub-solution} holds.
\end{remark}


\SE{Construction of solutions in large balls}\label{6}

We recall that $B_R(x)$ denotes the open Euclidean ball of $\R^N$ centered at $x\in \R^N$ and of radius $R>0$, and that $B_R=B_R(0)$. Throughout this section we suppose that $f$ and $J$ satisfy~\eqref{C1},~\eqref{C2} and~\eqref{C4}. Here, for any $R>0$ large enough and any $x_0\in \R^N$, we will construct and study the properties of positive continuous solutions of the following auxiliary problem
\begin{equation}\label{eq:BR}
\mathcal{L}_{B_R(x_0)}[v](x) -v(x)+f(v(x))=0 \quad \text{ for }x\in \overline{B_R(x_0)},
\end{equation}
where
\be\label{defL}
\ds{\mathcal{L}_{B_R(x_0)}[v](x):=\int_{B_R(x_0)}J(x-y)\,v(y)\,\mathrm{d}y}.
\ee
Besides the own interest of~\eqref{eq:BR}, the properties of some particular solutions $v$ of~\eqref{eq:BR} are essential in the proof of Theorem~\ref{TH:LIOUVILLE2}, as they will provide  key estimates  ensuring to derive the  asymptotic behaviour of the solutions $u$ of~\eqref{LIM:cu9} or~\eqref{LIM:u9}. So in Sections~\ref{sec61} and~\ref{sec62}, our main concern will be to establish, for any $x_0\in \R^N$ and $R>0$ large enough, the existence of a positive maximal solution $v_{x_0,R}$ to~\eqref{eq:BR}, such that $v_{x_0,R}\to 1$ locally uniformly in $\R^N$ as $R\to+\infty$. Based on the construction of these solutions in closed balls $\overline{B_R(x_0)}$, we will next show in Section~\ref{sec63} the existence of continuous and compactly supported sub-solutions in $\R^N$.


\subsection{Existence of a positive solution in $\overline{B_R(x_0)}$}\label{sec61}

This section is devoted to the proof of the existence of a positive continuous solution of~\eqref{eq:BR} in $\overline{B_R(x_0)}$, for any $R>0$ large enough and any $x_0 \in\R^N$.

\begin{lemma}\label{existence-BR}
Assume that $f$ and $J$ satisfy~\eqref{C1},~\eqref{C2} and~\eqref{C4}. Assume further that $J$ is compactly supported and $J\in L^2(\R^N)$. Then there exists $d_0=d_0(f,J)>0$ such that for every $x_0\in\R^N$ and $R\ge d_0$, problem~\eqref{eq:BR} admits a positive continuous solution $v:\overline{B_R(x_0)}\to(0,1)$ such that $\max_{\overline{B_R(x_0)}}v>\theta$, where $\theta\in(0,1)$ is defined in~\eqref{C4}.
\end{lemma}

\begin{proof}
Let $x_0\in \R^N$ be fixed, let also $R_J>0$ be fixed (independently of $x_0$) such that
$$\mathrm{supp}(J)\subset B_{R_J},$$
and pick any $R>R_J$. To construct a solution, we adapt the strategy used in~\cite{Clement} for the construction of a solution of a local reaction-diffusion equation. The proof is divided into three main steps.
\vskip 0.3cm

\noindent{\it{Step 1: definition and elementary properties of an energy functional~$\e$}}
\vskip 0.3cm
\noindent{}In the proof of Lemma~\ref{existence-BR}, let us extend $f$ by $f'(1)(s-1)$ for $s\ge 1$ and by  $-f(-s)$ for $s\le 0$ and denote $\tilde f$ this extension. Now, define
$$F(t):=\int_{0}^t\tilde f(s)\,\mathrm{d}s\ \hbox{ for }t\in\R,\ \ \ \ c(x):=1-\int_{B_R(x_0)}J(x-y)\,\mathrm{d}y\in[0,1]\ \hbox{ for }x\in\R^N,$$
and consider the following energy functional
\be \label{XEnergy}
\e(u)\!:=\!\frac{1}{4}\!\int_{\!B_R(x_0)}\!\int_{\!B_R(x_0)}\!\!\!\!J(x-y)\big(u(y)-u(x)\big)^{2}\mathrm{d}x\mathrm{d}y+\frac{1}{2}\!\int_{\!B_R(x_0)}\!\!\!\!c(x)u^2(x)\mathrm{d}x-\!\int_{\!B_R(x_0)}\!\!\!\!F(u(x))\,\mathrm{d}x
\ee
defined for $u\in L^2(B_R(x_0))$. Since $J\in L^1(\R^N)$, $\e$ is well defined in $L^2(B_R(x_0))$. Moreover, using the definition of $F$ and the oddness of $\tilde f$, we have
\begin{equation}\label{bcvh-eq-F}
\int_{B_R(x_0)} F(u(x))\,\mathrm{d}x = \int_{B_R(x_0)} F(|u(x)|)\,\mathrm{d}x
\end{equation}
for any $u\in L^2(B_R(x_0))$, while elementary computations yield
\be\label{bchv-def-energy1}
\e(u)=-\frac{1}{2}\int_{B_R(x_0)}\int_{B_R(x_0)}\!\!J(x\!-\!y)u(x)u(y)\,\mathrm{d}x\mathrm{d}y+\!\frac{1}{2}\int_{B_R(x_0)}\!\!u^{2}(x)\,\mathrm{d}x-\!\int_{B_R(x_0)}\!\!F(u(x))\,\mathrm{d}x.
\ee
{F}rom the last two formulas, one infers that, for any $u\in L^2(B_R(x_0))$,
$$\begin{array}{rcl}
\e(|u|) & \!\!\!=\!\!\! & \displaystyle-\frac{1}{2}\int_{B_R(x_0)}\!\int_{B_R(x_0)}\!\!\!J(x\!-\!y)|u(y)||u(x)|\,\mathrm{d}x\mathrm{d}y+\!\frac{1}{2}\!\int_{B_R(x_0)}\!\!\!|u(x)|^2\,\mathrm{d}x-\!\!\int_{B_R(x_0)}\!\!\!F(|u(x)|)\,\mathrm{d}x\vspace{3pt}\\
& \!\!\!\le\!\!\! & \displaystyle\e(u).
\end{array}$$

To complete Step 1, let us check that the functional $\e$ is bounded from below in $L^2(B_R(x_0))$. {F}rom~\eqref{C4} and~\eqref{bcvh-eq-F}, the definition of $F$ and $\tilde f$, and since  $\tilde f(s)\le 0$ for $s\ge 1$, we see that, for any $u\in L^2(B_R(x_0))$,
$$\int_{B_R(x_0)}F(u(x))\,\mathrm{d}x =\int_{B_R(x_0)}F(|u(x)|)\,\mathrm{d}x\le\int_{B_R(x_0)}\int_{0}^{\min\{1,|u(x)|\}}\tilde f(s)\,\mathrm{d}s \le R^N|B_1|\int_{0}^1f(s)\,\mathrm{d}s,$$
 where $|B_1|$ denotes the Lebesgue measure of the unit ball. Setting $C_0:=|B_1|\int_{0}^1f(s)\,ds>0$, we thus get that
\begin{equation}\label{bchv-eq-ener2}
\e(u)\!\ge\!\frac{1}{4}\!\int_{B_R(x_0)}\!\int_{B_R(x_0)}\!\!\!\!J(x-y)(u(y)-u(x))^{2}\,\mathrm{d}x\mathrm{d}y+\frac{1}{2}\!\int_{B_R(x_0)}\!\!\!c(x)u^2(x)\mathrm{d}x-C_0 R^N\!\ge\!-C_0 R^N
\end{equation}
for any $u\in L^2(B_R(x_0))$. Hence, the quantity
\be\label{defgamma}
\gamma:=\inf_{u\in L^2(B_R(x_0))} \e(u)
\ee
is a well defined real number.
\vskip 0.3cm

\noindent{\it{Step 2:  the infimum of $\e$ in $L^2(B_R(x_0))$ is achieved}}
\vskip 0.3cm
\noindent{}We shall now see that $\gamma$ is achieved for some $v\in L^2(B_R(x_0))$. So let $(u_n)_{n\in \N}$ be a minimising sequence. {F}rom the inequality $\e(|u|)\le\e(u)$, we may assume without loss of generality that the functions $u_n$ are all non-negative. Let us first check that the sequence $(u_n)_{n\in \N}$ is bounded in $L^2(B_R(x_0))$. To do so, we recall the definition~\eqref{defL} of $\mathcal{L}_{B_R(x_0)}$ and we notice that the principal eigenvalue  $\lambda_p$ of the operator $\mathcal{L}_{B_R(x_0)} - \mathrm{Id}$ is negative (see for example  \cite{Andreu,BCV2,Coville2010,GR} for a precise definition of $\lambda_p$ and some of its properties) and satisfies
$$-\lambda_p=\inf_{\|\varphi\|_{L^2(B_R(x_0))}=1} \left(\frac{1}{2} \int_{B_R(x_0)}\!\int_{B_R(x_0)}\!\!J(x-y)\big(\varphi(y)-\varphi(x)\big)^{2}\,\mathrm{d}x\mathrm{d}y+\int_{B_R(x_0)}\!\!c(x)\varphi^2(x)\,\mathrm{d}x\right).$$
As a consequence, from~\eqref{bchv-eq-ener2}, we get
$$\e(u_n)\ge -\frac{\lambda_p}{2}\int_{B_R(x_0)}u_n^2(x)\, \mathrm{d}x -C_0R^{N}$$
for all $n\in\N$. Therefore the sequence $(u_n)_{n\in \N}$ is bounded in $L^{2}(B_R(x_0))$ since it is a minimising sequence and since $\lambda_p<0$. Up to extraction of a subsequence, the sequence $(u_n)_{n\in \N}$ converges weakly in $L^2(B_R(x_0))$ to a non-negative function $v \in L^2(B_R(x_0))$.

We actually claim that
\begin{equation}\label{bchv-cla-enerv}
\e(v)=\gamma.
\end{equation}
Due to the lack of compactness in this non-local minimisation problem, we cannot expect to get a strong convergence in $L^2(B_R(x_0))$ for the minimising subsequence and therefore passing to the limit in the energy~\eqref{XEnergy} is not immediate. To overcome this difficulty, let us observe that by introducing the function
$$G(t):= \int_{0}^{t}\big(s -\tilde f(s)\big)\,ds=\frac{t^2}{2}-F(t),$$
we get from~\eqref{bchv-def-energy1} that, for any $n\in\N$,
$$\e(u_n)=-\frac{1}{2}\int_{B_R(x_0)}\int_{B_R(x_0)}J(x-y)\hspace{0.1em}u_n(x)\hspace{0.1em}u_n(y)\hspace{0.1em}\mathrm{d}x\mathrm{d}y +\int_{B_R(x_0)}G(u_n(x))\, \mathrm{d}x$$
and therefore
\begin{equation}\label{bchv-eq-lim2}\begin{array}{rcl}
\e(u_n) -\e(v) & = & \displaystyle-\frac{1}{2}\int_{B_R(x_0)}\int_{B_R(x_0)}J(x-y)\big[u_n(x)u_n(y)-v(x)v(y)\big]\,\mathrm{d}x\mathrm{d}y\vspace{3pt}\\
& & +\displaystyle\int_{B_R(x_0)}\big[G(u_n(x)) -G(v(x))\big]\,\mathrm{d}x.\end{array}
\end{equation}
We observe that the double integral ${\int_{B_R(x_0)}\int_{B_R(x_0)}J(x-y)u_n(x)u_n(y)\,\mathrm{d}x\mathrm{d}y}$ can be rewritten as
$$\begin{array}{rcl}
\displaystyle\int_{B_R(x_0)}\int_{B_R(x_0)}\!\!\!J(x\!-\!y)\hspace{0.1em}u_n(x)\hspace{0.07em}u_n(y)\mathrm{d}x\mathrm{d}y & \!\!\!\!=\!\!\!\! & \displaystyle\int_{B_R(x_0)}\!\!u_n(x)\left(\int_{B_R(x_0)}\!\!J(x\!-\!y)\big[u_n(y)\!-\!v(y)\big]\,\mathrm{d}y\right)\mathrm{d}x\vspace{3pt}\\
& & \displaystyle+ \int_{B_R(x_0)}\!v(y)\left(\int_{B_R(x_0)}\!J(x-y)u_n(x)\,\mathrm{d}x\right)\mathrm{d}y.
\end{array}$$
Using Lebesgue's dominated convergence theorem, together with the assumption $J\in L^2(\R^N)$ and the $L^2(B_R(x_0))$ weak convergence of the sequence $(u_n)_{n\in\N}$, it is easy to see that
$$\lim_{n\to+\infty}\int_{B_R(x_0)}\int_{B_R(x_0)}J(x-y)\hspace{0.1em}u_n(x)\hspace{0.1em}u_n(y)\hspace{0.1em}\mathrm{d}x\mathrm{d}y = \int_{B_R(x_0)}\int_{B_R(x_0)}J(x-y)\hspace{0.1em}v(x)\hspace{0.1em}v(y)\hspace{0.1em}\mathrm{d}x\mathrm{d}y$$
and therefore
\begin{equation}\label{bchv-eq-lim3}
\lim_{n\to+\infty} -\frac{1}{2}\int_{B_R(x_0)}\int_{B_R(x_0)}J(x-y)\big[u_n(x)u_n(y)-v(x)v(y)\big]\,\mathrm{d}x\mathrm{d}y = 0.
\end{equation}
On the other hand, since by assumption $\tilde f'(s)<1$ for all $s \in \R $, the function $G$ is convex and, for all $n\in\N$, we get
$$\int_{B_R(x_0)}\big[G(u_n(x)) -G(v)(x)\big]\,\mathrm{d}x\ge \int_{B_R(x_0)}G'(v(x))\big[u_n(x) -v(x)\big]\,\mathrm{d} x.$$
{F}rom the  definition of $G$ and $\tilde f$, together with the fact that $v\in L^2(B_R(x_0))$, we infer that $G'(v) \in L^2(B_R(x_0))$. Using the $L^2(B_R(x_0))$ weak convergence of $(u_n)_{n\in\N}$ to $v$, it follows that
\begin{equation}\label{bchv-eq-lim4}
\liminf_{n\to+\infty} \int_{B_R(x_0)}\big[G(u_n(x)) -G(v(x))\big]\, \mathrm{d}x \ge 0.
\end{equation}
Thus passing to the limit in~\eqref{bchv-eq-lim2}, and using~\eqref{bchv-eq-lim3} and~\eqref{bchv-eq-lim4}, we obtain $\gamma -\e(v)\ge 0$. Together with the definition~\eqref{defgamma} of $\gamma$, this shows that  $v$ is a minimiser of the energy $\e$, that is,~\eqref{bchv-cla-enerv} holds.
\vskip 0.3cm

\noindent{\it{Step 3: $v$ is a continuous positive solution $u$ of~\eqref{eq:BR}}}
\vskip 0.3cm
\noindent{}We first show in this step that $v$ is a solution to~\eqref{eq:BR} with $\tilde{f}$ instead of $f$. {F}rom~\eqref{bchv-cla-enerv}, $v$ is a critical point of $\e$ and in particular, it follows from the formulation~\eqref{bchv-def-energy1} of $\e$ that $v$ is a non-negative weak solution of $\mathcal{L}_{B_R(x_0)}[v] -v +\tilde f(v)=0$ in $B_R(x_0)$. Since all functions $\mathcal{L}_{B_R(x_0)}[v]$, $v$ and $\tilde f(v)$ belong to $L^2(B_R(x_0))$, the function $v$ satisfies $\mathcal{L}_{B_R(x_0)}[v](x)-v(x)+\tilde f(v(x))=0$ for a.e. $x\in B_R(x_0)$. Furthermore, since $J\in L^2(\R^N)$, the Cauchy-Schwarz inequality implies that $\mathcal{L}_{B_R(x_0)}[v]\in L^{\infty}(B_R(x_0))$. Therefore, since by~\eqref{C4} the function $s\mapsto s-\tilde{f}(s)$ is a $C^1$ diffeomorphism from $\R^+$ onto $\R^+$ and since $v$ is non-negative, it follows from the equation $\mathcal{L}_{B_R(x_0)}[v]-v+\tilde f(v)=0$ a.e. in $B_R(x_0)$ that $v\in L^{\infty}(B_R(x_0))$.  Thus by reproducing the arguments of the proof of Lemma~\ref{PROP:CONE} we  deduce that $v$ has a representative, still denoted $v$, which is continuous in $\overline{B_R(x_0)}$ and satisfies
\begin{equation}\label{eq:BRtilde-fort}
\mathcal{L}_{B_R(x_0)}[v](x)-v(x)+\tilde f(v(x))=0 \quad{\mbox{ for all }}x\in\overline{B_R(x_0)}.
\end{equation}

Remember now that, from~\eqref{C2}, $J>0$ a.e. in $\mathcal{A}(r_1,r_2)$ with $0\le r_1<r_2$, and that $R\ge R_J\ge r_2>r_1$, with $\mathrm{supp}(J)\subset B_{R_J}$. As a consequence, if
there exists a point $x\in\overline{B_R(x_0)}$ such that $v(x)=0$, then, arguing as in the proof of the strong maximum principle (Lemma~\ref{STRONG:LE23}) or as in the proof of
the sweeping principle (Lemma~\ref{TH:SWEEP}), it follows that $v=0$ in $\overline{\mathcal{A}(x,r_1,r_2)}\cap\overline{B_R(x_0)}$, hence $v=0$ in $\overline{\mathcal{A}(y,r_1,r_2)}\cap\overline{B_R(x_0)}$ for all $y\in\overline{\mathcal{A}(x,r_1,r_2)}\cap\overline{B_R(x_0)}$ and finally $v=0$ in $\overline{B_r(x)}\cap\overline{B_R(x_0)}$ for some $r>0$. Therefore, the non-empty set $\big\{x\in\overline{B_R(x_0)};\ v(x)=0\big\}$ is both (obviously) closed and open relatively to $\overline{B_R(x_0)}$ and is thus equal to $\overline{B_R(x_0)}$. As a consequence, either $v\equiv 0$ in $\overline{B_R(x_0)}$ or $v>0$ in $\overline{B_R(x_0)}$.

In this paragraph, we prove that the solution $v$ constructed is a solution of~\eqref{eq:BR}, namely we just need to show that $v\le 1$ in $\overline{B_R(x_0)}$. To do so, define $M=\max_{\overline{B_R(x_0)}}v\ge 0$ and let $\bar{x}\in\overline{B_R(x_0)}$ be such that $v(\bar{x})=M$. Assume by contradiction that $M>1$. By evaluating~\eqref{eq:BRtilde-fort} at $\bar{x}$ and using the definition of $\tilde f$, we get that
$$\int_{B_R(x_0)}J(\bar{x}-y)\hspace{0.1em}v(y)\hspace{0.1em}\mathrm{d}y=\mathcal{L}_{B_R(x_0)}[v](\bar{x})=M-\tilde f(M)>M.$$
Since $v\le M$ in $\overline{B_R(x_0)}$, this leads to a contradiction. Hence $M\le 1$ and thus $v$ is a non-negative continuous solution of~\eqref{eq:BR} in $\overline{B_R(x_0)}$. Furthermore, as for the positivity of $v$, one gets that either $v\equiv 1$ in $\overline{B_R(x_0)}$ or $v<1$ in $\overline{B_R(x_0)}$. The former case is impossible since $\mathcal{L}_{B_R(x_0)}[v]\not\equiv 1$ in $\overline{B_R(x_0)}$ (indeed, $\int_{B_R(x_0)}J(x-y)\,\mathrm{d}y<1$ for all $x\in\partial B_R(x_0)$). Thus, $0\le v<1$ in $\overline{B_R(x_0)}$.

Finally, let us verify that the solution $v$ constructed is not the trivial solution. To do so, it is enough to show that $\e(v)\neq\e(0)=0$. We claim that, for $R>R_J$ large enough, $\e(v)<0$.  Indeed, let us consider the test function $\varphi:=\mathds{1}_{B_R(x_0)}\in L^2(B_R(x_0))$. We have
$$\begin{array}{rcl}
\e(\varphi) & \!\!\!=\!\!\! & \displaystyle\frac{1}{4}\int_{\!B_R(x_0)}\!\int_{\!B_R(x_0)}\!\!\!\!J(x\!-\!y)\big(\varphi(y)\!-\!\varphi(x)\big)^{2}\,\mathrm{d}x\mathrm{d}y+\frac{1}{2}\!\int_{\!B_R(x_0)}\!\!\!\!c(x)\varphi^2(x)\,\mathrm{d}x -\!\int_{\!B_R(x_0)}\!\!\!\!F(\varphi(x))\,\mathrm{d}x\vspace{3pt}\\
& \!\!\!=\!\!\! & \displaystyle\frac{1}{2}\int_{B_R(x_0)}c(x)\,\mathrm{d}x-R^N|B_1|\int_{0}^{1}f(s)\,\mathrm{d}s\vspace{3pt}\\
& \!\!\!=\!\!\! & \displaystyle\frac{1}{2}\int_{B_R(x_0)}\int_{\R^N\setminus B_R(x_0)}J(x-y)\,\mathrm{d}y\,\mathrm{d}x-R^N|B_1|\int_{0}^{1}f(s)\,\mathrm{d}s.
\end{array}$$
Since $\mathrm{supp}(J)\subset B_{R_J}$, the above equality yields
\begin{align*}
\e(\varphi)&=\frac{1}{2}\int_{B_R(x_0)\setminus B_{R-R_J}(x_0)}\int_{\R^N\setminus B_R(x_0)}J(x-y)\,\mathrm{d}y\,\mathrm{d}x-R^N|B_1|\int_{0}^{1}f(s)\,\mathrm{d}s,\\
&\le \frac{1}{2}\,|B_1|\left(R^N-(R-R_J)^N\right) - R^N|B_1|\int_{0}^{1}f(s)\,\mathrm{d}s.
\end{align*}
Thus, since $\int_{0}^1f(s)\,ds>0$, there exists $d_0=d_0(J,f)\in(R_J,+\infty)$, independent of $x_0$, such that, for every $R\ge d_0$, the right-hand side of the above inequality is negative and thus $\e(v)\le \e(\varphi)<0,$ which proves our claim. Furthermore, since $0\le v<1$ in $\overline{B_R(x_0)}$ and $F\le0$ in $[0,\theta]$, one infers that $\max_{\overline{B_R(x_0)}}v>\theta$, hence $v>0$ in $\overline{B_R(x_0)}$ (remember that $v$ was either positive or identically equal to $0$ in $\overline{B_R(x_0)}$).

As a conclusion, for every $R$ such that $R\ge d_0$, there exists a solution $v\in C(\overline{B_R(x_0)},(0,1))$ to~\eqref{eq:BR} with $\max_{\overline{B_R(x_0)}}v>\theta$. The point $x_0\in\R^N$ being  arbitrary and the constant $d_0$ being independent of $x_0$, the proof of Lemma~\ref{existence-BR} is thereby complete.
\end{proof}


\subsection{Existence and properties of the maximal solution in $\overline{B_R(x_0)}$}\label{sec62}

Let us now look more closely at the properties of positive solutions of~\eqref{eq:BR} and in particular at the maximal solution, if any. To this end,  let us in this subsection extend continuously $f$ by $f'(1)(s-1)$ for $s\ge 1$ and by $0$ for $s\le 0$. To simplify our presentation let us still denote $f$ this extension.

Let us first recall the notion of maximal solution for problem~\eqref{eq:BR}.

\begin{definition}
Let $x_0\in \R^N$ and $R>0$. A function $v\in C(\overline{B_R(x_0)},[0,1])$ is called a \emph{maximal solution} to~\eqref{eq:BR} in  $\overline{B_R(x_0)}$ if any solution $w\in C(\overline{B_R(x_0)},[0,1])$ satisfies $w\le v$ in $\overline{B_R(x_0)}$.
\end{definition}

The following lemma provides the existence and uniqueness of a maximal solution to the problem~\eqref{eq:BR} when $R>0$ is large enough.

\begin{lemma}\label{bchv-lem-exis-msol}
Assume that $f$ and $J$ satisfy~\eqref{C1},~\eqref{C2} and~\eqref{C4}. Assume further that $J$ is compactly supported and $J\in L^2(\R^N)$.  Then there exists $d_0=d_0(f,J)>0$, given as in Lemma~$\ref{existence-BR}$, such that for every $x_0\in \R^N$ and $R\ge d_0$, problem~\eqref{eq:BR} admits a unique maximal solution $v_{x_0,R}$ in $\overline{B_{R}(x_0)}$ and $v_{x_0,R}$ satisfies $0<v_{x_0,R}<1$ in $\overline{B_R(x_0)}$.
\end{lemma}

\begin{proof}
Let  $x_0\in \R^N$ be fixed and let $R$ be fixed such that $R\ge d_0$, where $d_0=d_0(f,J)>0$ is given in Lemma~\ref{existence-BR}. We will check that the conclusion holds with this quantity $d_0$. First of all, the uniqueness of the maximal solution in $\overline{B_R(x_0)}$, if any, is a trivial consequence of its definition.

Let us then focus on the construction of a maximal solution.  {F}rom Lemma~\ref{existence-BR}, there exists a solution $v\in C(\overline{B_{R}(x_0)},(0,1))$ to~\eqref{eq:BR}. Now, remember that $1$ is a strict super-solution to~\eqref{eq:BR}. Therefore, since $f$ is Lipschitz continuous, it follows that we can construct a maximal solution $v_{x_0,R}\in C(\overline{B_{R}(x_0)},(0,1))$ to~\eqref{eq:BR} such that
$$0<v\le v_{x_0,R}<1\ \hbox{ in }\overline{B_R(x_0)}$$
by using  standard monotone iterative scheme as in~\cite[Theorem~A.1]{Coville2010}. For the sake of completeness, let us describe this scheme in the next paragraph.

First, let us observe that, from the assumptions on $J$, the linear operator $\mathcal{L}_{B_R(x_0)}$ is a continuous operator in $C(\overline{B_R(x_0)})$. Next let us choose a real number $k>0$ large enough such that the function $s\mapsto-ks-f(s)$ is decreasing in $\R$. We can increase further $k$ if necessary to ensure that $k+1\in \rho(\mathcal{L}_{B_R(x_0)})$, where $\rho(\mathcal{L}_{B_R(x_0)}) $ denotes the resolvent of the operator $\mathcal{L}_{B_R(x_0)}$. 
We note that, by this choice of $k$, the operator $\mathcal{L}_{B_R(x_0)}- (k+1)$ satisfies a comparison principle, in the sense that if $w\in C(\overline{B_R(x_0)})$ satisfies $\mathcal{L}_{B_R(x_0)}[w]- (k+1)w\ge0$ in $\overline{B_R(x_0)}$ then $w\le0$ in $\overline{B_R(x_0)}$ (see \cite{Coville2010, Coville2017}). Now, set $v_0=1$ and let $v_1\in C(\overline{B_R(x_0)})$ be the solution of the following linear problem
\begin{equation}\label{bchv-eq-induc1}
\mathcal{L}_{B_R(x_0)} [v_1](x) -(k+1)v_1(x)=-kv_0(x)-f(v_0(x)) \quad\text{ for }x\in\overline{B_R(x_0)}.
\end{equation}
The function $v_1$ is well defined, since by construction the continuous operator $\mathcal{L}_{B_R(x_0)}-(k+1)$ is invertible. We claim that $v \le v_1 \le v_0$ in $\overline{B_R(x_0)}$. Indeed, since $v\,(\le1)$ and $v_0=1$ are respectively a solution and a super-solution of~\eqref{eq:BR}, we have, for $x\in\overline{B_R(x_0)}$,
\begin{align*}
&\mathcal{L}_{B_R(x_0)}[v_1-v_0](x)-(k+1)(v_1(x)-v_0(x))=-\mathcal{L}_{B_R(x_0)}[1](x)+1\ge 0,\vspace{3pt}\\
&\mathcal{L}_{B_R(x_0)}[v_1-v](x)-(k+1)(v_1(x)-v(x))=-kv_0(x)-f(v_0(x))+kv(x)+f(v(x))\le 0.
\end{align*}
So, the inequality $v\le v_1 \le v_0$ in $\overline{B_R(x_0)}$ follows from the comparison principle satisfied by the operator $\mathcal{L}_{B_R(x_0)}-(k+1)$. In particular, $0<v_1\le 1$ in $\overline{B_R(x_0)}$. Now let  $v_2\in C(\overline{B_R(x_0)})$ be  the solution of~\eqref{bchv-eq-induc1} with $v_2$ instead of $v_1$ in the left-hand side and $v_1$ instead of $v_0$ in the right-hand side.  {F}rom the monotonicity of $s\mapsto-ks-f(s)$ and from the  comparison principle, we have $v\le v_2\le v_1\le v_0$ in $\overline{B_R(x_0)}$. By induction, we can construct a non-increasing  sequence of functions $(v_n)_{n\in\N}$ in $C(\overline{B_R(x_0)})$ satisfying  $v\le v_{n+1}\le v_n \le v_0$ in $\overline{B_R(x_0)}$ and
\begin{equation}\label{pev-eq-induc}
\mathcal{L}_{B_R(x_0)}[v_{n+1}](x) -(k+1)v_{n+1}(x)=-kv_n(x)-f(v_n(x)) \quad\text{ for }x\in\overline{B_R(x_0)}.
\end{equation}
Since the sequence is non-increasing and bounded from below, the quantity
$$v_{x_0,R}(x):=\inf_{n\in \N}v_n(x)=\lim_{n\to+\infty}v_n(x)\ \in[v(x),1]\ (\subset(0,1])$$
is well defined for every $x\in\overline{B_R(x_0)}$.  Moreover, by passing to the limit in the equation~\eqref{pev-eq-induc} and using Lebesgue's dominated convergence theorem, it follows that  $v_{x_0,R}$ is a solution of~\eqref{eq:BR}. As in the proof of Lemma~\ref{PROP:CONE}, we infer that $v_{x_0,R}$ is continuous in $\overline{B_R(x_0)}$ and, as in the proof of Lemma~\ref{existence-BR}, we get that $v_{x_0,R}<1$ in $\overline{B_R(x_0)}$. To sum up, $v_{x_0,R}$ is a solution of~\eqref{eq:BR} belonging to $C(\overline{B_{R}(x_0)},(0,1))$.

We finally claim that $v_{x_0,R}$ is a maximal solution to~\eqref{eq:BR}. Indeed, let $w\in C(\overline{B_{R}(x_0)},[0,1])$ be any solution to~\eqref{eq:BR}. By replacing $v$ with $w$ in the arguments of the previous paragraph and using the fact that the sequence $(v_n)_{n\in\N}$ is defined with the same initial value $v_0=1$, we get that $w\le v_n$ in $\overline{B_R(x_0)}$ for every $n\in\N$, hence $w\le v_{x_0,R}$ in $\overline{B_R(x_0)}$. The proof of Lemma~\ref{bchv-lem-exis-msol} is thereby complete.
\end{proof}

The maximal solutions $v_{x_0,R}$ possess some important properties, in particular they are monotone non-decreasing with respect to the domains.

\begin{lemma}\label{prop:VR} Let us assume that $f$ and $J$ satisfy~\eqref{C1},~\eqref{C2} and~\eqref{C4}. Assume further that $J$ is compactly supported and $J\in L^2(\R^N)$. Let $d_0=d_0(f,J)>0$ be given as in Lemmas~$\ref{existence-BR}$ and~$\ref{bchv-lem-exis-msol}$. The following properties hold:
\begin{itemize}
\item[(i)] for every $x_{1},x_{2} \in \R^N$ and $d_0\le R_{1}\le R_{2}$ such that $B_{R_{1}}(x_{1})\subset B_{R_{2}}(x_{2})$, then
$$v_{x_{1},R_{1}}(x)\le v_{x_{2},R_{2}}(x) \quad \text{for all }x\in\overline{ B_{R_{1}}(x_{1})};$$
\item[(ii)] for every $x_{0}\in \R^N$ and $R\ge d_0$, the function $v_{0,R}(\cdot -x_{0})$ defined in $\overline{B_{R}(x_{0})}$ satisfies $v_{0,R}(\cdot-x_0)=v_{x_{0},R}$ in $\overline{B_R(x_0)}$;
\item[(iii)] for every $x_0\in \R^N$ and $R\ge d_0$,
$$ \min_{\overline{B_R(x_0)}}v_{x_0,4R}\ge \max_{\overline{B_R(x_0)}}v_{x_0,2R}.$$
\end{itemize}
\end{lemma}

\begin{proof}
The proof of (i) is straightforward. Indeed, let $x_{1},x_{2} \in \R^N$ and $d_0\le R_{1}\le R_{2}$ be such that $B_{R_{1}}(x_{1})\subset B_{R_{2}}(x_{2})$.  Recall from the proof of Lemma~\ref{bchv-lem-exis-msol} that $v_{x_{2},R_{2}}\in C(\overline{B_{R_2}(x_2)},(0,1))$ can be defined as $v_{x_{2},R_{2}}=\lim_{n\to+\infty}v_n$, where $(v_n)_{n\in\N}$ is the sequence of positive functions in $C(\overline{B_{R_2}(x_2)},(0,1])$ defined by induction by $v_0=1$ in $\overline{B_{R_2}(x_2)}$ and, for $n\in\N$,
$$ \mathcal{L}_{B_{R_{2}}(x_{2})}[v_{n+1}](x) -(k+1)v_{n+1}(x)=-kv_n(x)-f(v_n(x)) \quad\text{ for }x\in\overline{ B_{R_{2}}(x_{2})}.$$
Here $k>0$ is such that $k+1\in \rho(\mathcal{L}_{B_{R_{2}}(x_{2})})$ and the function $s\mapsto-ks-f(s)$ is decreasing. By increasing $k$ if necessary we may assume that $k+1\in \rho(\mathcal{L}_{B_{R_{2}}(x_{2})}) \cap \rho(\mathcal{L}_{B_{R_{1}}(x_{1})})$. Now observe that, for any $n\in\N$, $v_n$ satisfies
\begin{equation}\label{pev-eq-inducsup}
\mathcal{L}_{B_{R_{1}}(x_{1})}[v_{n+1}](x) -(k+1)v_{n+1}(x)\le -kv_n(x)-f(v_n(x)) \quad\text{ for }x \in \overline{ B_{R_{1}}(x_{1})},
\end{equation}
that is, the function $v_{n+1}$ is a super-solution to problem~\eqref{pev-eq-induc} in $\overline{ B_{R_{1}}(x_{1})}$. We  claim that, for every $n\in\N$,
$$v_{x_{1},R_{1}}(x)\le v_n(x)\ \hbox{ for all }x\in\overline{ B_{R_{1}}(x_{1})}.$$
To do so, we proceed by induction. By construction of $v_{x_1,R_1}$ and the definition of $v_0$, we know that $v_{x_{1},R_{1}}(x)\le v_0(x)$  for all $x \in \overline{ B_{R_{1}}(x_{1})}$. For $n\in\N$, assume that $v_{x_{1},R_{1}}(x)\le v_n(x)$ for all $x \in \overline{ B_{R_{1}}(x_{1})}$, and let us prove that $v_{x_{1},R_{1}}(x)\le v_{n+1}(x)$ for all $x \in \overline{ B_{R_{1}}(x_{1})}$. Let $w:=v_{x_{1},R_{1}}-v_{n+1}$ in $\overline{B_{R_1}(x_1)}$. {F}rom~\eqref{pev-eq-inducsup}, since the function $s\mapsto-ks-f(s)$ is decreasing and $v_{x_{1},R_{1}}(x)\le v_n(x)$ for all $x \in \overline{ B_{R_{1}}(x_{1})}$, we see that $w$ satisfies
$$\mathcal{L}_{B_{R_{1}}(x_{1})}[w](x)\!-\!(k\!+\!1)w(x)\!\ge\! -kv_{x_{1},R_{1}}(x)\!-\!f(v_{x_{1},R_{1}}(x))\!+\!kv_n(x)\!+\!f(v_n(x))\!\ge\!0\hbox{ for }x\!\in\!\overline{B_{R_1}(x_1)}.$$
Since the operator $\mathcal{L}_{B_{R_1}(x_1)}-(k+1)$ satisfies the maximum principle we then deduce that $w\le0$ in $\overline{B_{R_1}(x_1)}$, that is, $v_{x_{1},R_{1}}(x)\le v_{n+1}(x)$ for all $x \in \overline{B_{R_{1}}(x_{1})}$.  Therefore, for every $x \in \overline{ B_{R_{1}}(x_{1})}$, we have $v_{x_{1},R_1}(x)\le \lim_{n\to+\infty} v_n(x)=v_{x_{2},R_{2}}(x)$.

Part (ii) follows from the following observations. For any $x_0\in\R^N$, the function $v_{0,R}(\cdot-x_0)\in C(\overline{B_R(x_0)},(0,1))$ satisfies
$$\mathcal{L}_{B_R(x_0)}[v_{0,R}(\cdot-x_0)](x)-v_{0,R}(x-x_0)+f(v_{0,R}(x-x_0))=0\ \hbox{ for all }x\in\overline{B_R(x_0)}.$$
Therefore, by the maximality of $v_{x_0,R}$, it follows that $v_{0,R}(\cdot-x_0)\le v_{x_0,R}$ in $\overline{B_R(x_0)}$. Similarly, one can show that $v_{x_0,R}(\cdot+x_0)\le v_{0,R}$ in $\overline{B_R}$. Finally, $v_{0,R}(\cdot-x_0)=v_{x_0,R}$ in $\overline{B_R(x_0)}$.

To prove (iii), we simply observe that, for any $x_1\in\overline{B_{2R}(x_0)}$, one has $B_{2R}(x_1)\subset B_{4R}(x_0)$ and, by~(i), $v_{x_0,4R}\ge v_{x_1,2R}$ in $\overline{B_{2R}(x_1)}$. Property~(ii) yields $v_{x_0,2R}(\cdot -(x_1-x_0))=v_{x_1,2R}$ in $\overline{B_{2R}(x_1)}$, hence
$$v_{x_0,4R}(x)\ge v_{x_0,2R}(x-(x_1-x_0)) \quad \text{for all }x_1\in\overline{B_{2R}(x_0)}\hbox{ and }x\in\overline{B_{2R}(x_1)}.$$
Now, since for every $x,y \in\overline{B_R(x_0)}$ there exists (a unique) $x_1\in\overline{B_{2R}(x_0)}$ such that $y=x-(x_1-x_0)$ and $x\in\overline{B_R(x_1)}\subset\overline{B_{2R}(x_1)}$, the latter inequality implies that
$$v_{x_0,4R}(x)\ge v_{x_0,2R}(x-(x_1-x_0))=v_{x_0,2R}(y)$$
for all $x,y\in\overline{B_R(x_0)}$, which completes the proof.
\end{proof}

We can now state our last property about the maximal solution.

\begin{lemma}\label{prop:limVR} Let us assume that $f$ and $J$ satisfy~\eqref{C1},~\eqref{C2}  and~\eqref{C4}. Assume further that $J$ is compactly supported and $J\in L^{2}(\R^N)$. Then, for every $x_0\in\R^N$, $v_{x_0,R}\to1$ as $R\to+\infty$ locally uniformly in $\R^N$.
\end{lemma}

\begin{proof}
Let $x_0\in \R^N$ be fixed. Consider any non-decreasing sequence $(R_n)_{n\in \N}$ in $[d_0,+\infty)$ and converging to $+\infty$, where $d_0=d_0(f,J)>0$ is given in Lemmas~\ref{existence-BR} and~\ref{bchv-lem-exis-msol} (we recall that $d_0>R_J$, where $\supp(J)\subset B_{R_J}$). Thanks to part~(i) of Lemma~\ref{prop:VR}, the sequence  $(v_{x_0,R_n})_{n\in\N}$ is non-decreasing, in the sense that $v_{x_0,R_n}\le v_{x_0,R_p}$ in $\overline{B_{R_n}(x_0)}$ for all $n\le p$. Moreover, $0<v_{x_0,R_n}<1$ in $\overline{B_{R_n}(x_0)}$ for each $n\in\N$. As a consequence, the sequence $(v_{x_0,R_n})_{n\in\N}$ converges pointwise in $\R^N$ to a function $0<\bar v\le 1$ which, thanks to Lebesgue's dominated convergence theorem, satisfies
\begin{equation}\label{eq:global}
J\ast \bar v(x) -\bar v(x) +f(\bar v(x))=0 \quad{\mbox{ for all }}x\in\R^N.
\end{equation}
As in the proof of Lemma~\ref{PROP:CONE}, the function $\bar{v}$ can be viewed as a uniformly continuous function and therefore the limit $v_{x_0,R_n}\to\bar{v}$ holds locally uniformly in $\R^N$.

Consider now any $x_1\in\R^N$ and any $\delta\in[d_0,+\infty)$. We can then extract a subsequence of $(R_n)_{n\in \N}$ that we still denote $(R_n)_{n\in \N}$ such that, for all $n\in\N$, $B_{\delta}(x_1)\subset B_{R_n}(x_0)$ and $R_{n+1}\ge 4R_{n}$. By Lemma~\ref{existence-BR} and parts~(i) and~(iii) of Lemma~\ref{prop:VR}, we get that
\be\label{ineqs}\baa{l}
\ds1>\min_{\overline{B_\delta(x_1)}} v_{x_0,R_{n+1}}\ge \min_{\overline{B_{R_n}(x_0)}} v_{x_0,R_{n+1}}\ge\min_{\overline{B_{R_n}(x_0)}}v_{x_0,4R_n}\ge\max_{\overline{B_{R_n}(x_0)}}v_{x_0,2R_n}\ge\cdots\vspace{3pt}\\
\qquad\qquad\ds\cdots\ge\max_{\overline{B_{R_n}(x_0)}}v_{x_0,R_n}\ge\max_{\overline{B_\delta(x_1)}}v_{x_0,R_n}\ge\max_{\overline{B_\delta(x_1)}}v_{x_1,\delta}>\theta.\eaa
\ee
Taking the limit as $n\to+\infty$ in the inequality $$
\min_{\overline{B_\delta(x_1)}} v_{x_0,R_{n+1}}\ge\max_{\overline{B_\delta(x_1)}}v_{x_0,R_n},$$
we obtain that
$$\min_{\overline{B_{\delta}(x_1)}}\bar v \ge \max_{\overline{B_{\delta}(x_1)}}\bar v.$$ Hence, $\bar{v}$ is equal to a constant $C_{x_1,\delta}$ in $\overline{B_{\delta}(x_1)}$ and, thanks to~\eqref{ineqs}, there holds $\theta<C_{x_1,\delta}\le1$. Furthermore,
since~$x_1\in\R^N$ is
arbitrary, it follows that $\bar{v}$ is equal to a constant $C\in(\theta,1]$ in $\R^N$.

Lastly,~\eqref{eq:global} yields $f(C)=0$. Since $f$ satisfies~\eqref{C4} and $\theta<C\le1$, we infer that $C=1$. Therefore, $\bar v=1$ in $\R^N$ and thus the sequence $(v_{x_0,R_n})_{n\in\N}$ converges to $1$ locally uniformly in $\R^N$ as $n\to+\infty$. Since
the non-decreasing sequence $(R_n)_{n\in \N}$ converging to $+\infty$ is arbitrary,
and so is~$\delta\in [d_0,+\infty)$, it follows that $v_{x_0,R}$ converges to $1$ locally uniformly in $\R^N$ as $R\to+\infty$. The proof of Lemma~\ref{prop:limVR} is thereby complete.
\end{proof}


\subsection{Compactly supported continuous sub-solutions in $\R^N$}\label{sec63}

In this section, we construct compactly supported continuous sub-solutions from $\R^N$ to $[0,1]$ of problems of the type~\eqref{eq:BR}. Such continuous sub-solutions will then serve as a building block of some lower bounds in the proof of Theorem~\ref{TH:LIOUVILLE2}.

Let us first introduce some useful notations. For $x_0\in\R^N$, $R>0$ and $x\in\R^N$, let $\mathcal{P}_{x_0,R}(x)$ be the projection of $x$ on the closed convex set $\overline{B_R(x_0)}$, that is, $\mathcal{P}_{x_0,R}(x)\in \overline{B_R(x_0)}$ and
$$|x-\mathcal{P}_{x_0,R}(x)|=\mathrm{dist}(x,B_R(x_0))=\min_{y\in\overline{B_R(x_0)}}|x-y|.$$

\begin{lemma}\label{cla:subsol}
Assume that $f$ and $J$ satisfy~\eqref{C1},~\eqref{C2} and~\eqref{C4}. Assume further that $J$ is compactly supported and $J\in L^2(\R^N)$. Let $d_0=d_0(f,J)>0$ be given as in Lemmas~$\ref{existence-BR}$ and~$\ref{bchv-lem-exis-msol}$ and, for any $x_0\in \R^N$ and $R\ge d_0$, let $v_{x_0,R}\in C(\overline{B_R(x_0)},(0,1))$ be the maximal solution of~\eqref{eq:BR}. Then
there exists~$\delta_0>0$ such that, for any $x_0\in \R^N$, $R\ge d_0$ and $\delta\in(0,\delta_0]$, the continuous function $w_{x_0,R,\delta}:\R^N\to[0,1)$ defined by
\be\label{AX}
w_{x_0,R,\delta}(x)=\max\big\{v_{x_0,R}(\p{x_0,R}(x))-\delta^{-1}\,|x-\mathcal{P}_{x_0,R}(x)|,0\big\}\ee
satisfies
\be\label{subsol}
\underbrace{\int_{B_{R'}(x_0)}\!\!\!J(x-y)\,w_{x_0,R,\delta}(y)\,\mathrm{d}y}_{=\mathcal{L}_{B_{R'}(x_0)}[w_{x_0,R,\delta}](x)} -w_{x_0,R,\delta}(x)+f(w_{x_0,R,\delta}(x))\ge 0 \quad \text{for all }x\in\R^N
\ee
and for all $R'\ge R+\delta$.
\end{lemma}

\begin{proof} In view of~\eqref{AX},
we see that,
for every $x_0\in\R^N$, $R\ge d_0$ and $\delta>0$, the function $w_{x_0,R,\delta}$ is continuous $\R^N$, that $0\le  w_{x_0,R,\delta}<1$ in $\R^N$, that $w_{x_0,R,\delta}=v_{x_0,R}$ in $\overline{B_R(x_0)}$ and that $w_{x_0,R,\delta}=0$ in $\R^N\setminus B_{R+\delta}(x_0)$.\par
We set $g(s):=s-f(s)$ for $s\in[0,1]$.
{F}rom~\eqref{C4}, 
we see that
\begin{equation}\label{eq:gamma}
\gamma:=\min_{[0,1]}g'\,>0.
\end{equation}
We recall that, by~\eqref{C4}, $J$ is assumed to belong to $W^{1,1}(\R^N)$, and set
\be\label{defdelta0}
\delta_0:=\gamma\times\Big(\int_{\R^N}|\nabla J(z)|\,\mathrm{d}z\Big)^{-1}>0.
\ee\par
Let us now fix any $x_0\in\R^N$, $R\ge d_0$, $\delta\in(0,\delta_0]$ and let us check that~\eqref{subsol} holds for any $R'\ge R+\delta$. Since both $w_{x_0,R,\delta}$ and $J$ are non-negative and since $w_{x_0,R,\delta}=0$ in $\R^N\setminus B_{R+\delta}(x_0)$,
recalling also that $f(0)=0$ due to~\eqref{C4}, 
we see that
it is sufficient to show~\eqref{subsol} for $x\in B_{R+\delta}(x_0)$. Furthermore, by monotonicity of the integral with respect to $R'$, it is enough to show~\eqref{subsol} for $R'=R+\delta$.\par
For any $x\in B_{R+\delta}(x_0)$, there holds
$$\int_{B_{R+\delta}(x_0)}\!\!\!J(x-y)\,w_{x_0,R,\delta}(y)\,\mathrm{d}y=\int_{B_{R+\delta}(x_0)\setminus B_R(x_0)}\!\!\!J(x-y)\,w_{x_0,R,\delta}(y)\,\mathrm{d}y +\int_{B_{R}(x_0)}\!\!\!J(x-y)\,v_{x_0,R}(y)\,\mathrm{d}y.$$
Therefore, it follows from the above equality and the definitions of $v_{x_0,R}$ and $w_{x_0,R,\delta}$ that, for $x\in\overline{B_{R}(x_0)}$,
$$\int_{B_{R\!+\!\delta}(x_0)}\!\!\!J(x-y)\,w_{x_0,R,\delta}(y)\,\mathrm{d}y-w_{x_0,R,\delta}(x)+f(w_{x_0,R,\delta}(x))\!=\!\int_{B_{R\!+\!\delta}(x_0)\setminus B_R(x_0)}\!\!\!\!\!J(x\!-\!y)\,w_{x_0,R,\delta}(y)\,\mathrm{d}y\!\ge\!0.$$\par
To complete our proof, we have to show that the above inequality holds also for $x\in B_{R+\delta}(x_0)\setminus \overline{B_R(x_0)}$. To this end, let us consider $x\in B_{R+\delta}(x_0)\setminus \overline{B_R(x_0)}$ and set
$$s(x):=v_{x_0,R}(\p{x_0,R}(x))\ \hbox{ and }\ \tau(x):= \mathrm{dist}(x,B_R(x_0))=|x-\p{x_0,R}(x)|>0,$$
that is, $w_{x_0,R,\delta}(x)=\max\{s(x)-\delta^{-1}\tau(x),0\}$. {F}rom the nonnegativity of $J$ and $w_{x_0,R,\delta}$ and the fact that $w_{x_0,R,\delta}=v_{x_0,R}$ in $\overline{B_R(x_0)}$, we have
\be\label{eq:subsol2}\baa{l}
\ds\int_{B_{R\!+\!\delta}(x_0)}\!\!\!J(x-y)\,w_{x_0,R,\delta}(y)\,\mathrm{d}y-w_{\delta,R,x_0}(x)+f(w_{\delta,R,x_0}(x))\vspace{3pt}\\
\quad\ds\ge \int_{B_{R}(x_0)}\!\!\!J(x\!-\!y)\,v_{x_0,R}(y)\,\mathrm{d}y\!-\!\max\{s(x)\!-\!\delta^{-1}\tau(x),0\}\!+\!f(\max\{s(x)\!-\!\delta^{-1}\tau(x),0\}).
\eaa\ee
Now, two situations may occur: either $s(x)\le\delta^{-1}\tau(x)$ (that is, $w_{x_0,R,\delta}(x)=0$), or $s(x)>\delta^{-1}\tau(x)$ (that is, $w_{x_0,R,\delta}(x)>0$). In the first situation, we easily conclude that
$$\int_{B_{R\!+\!\delta}(x_0)}\!\!\!J(x-y)\,w_{x_0,R,\delta}(y)\,\mathrm{d}y-w_{x_0,R,\delta}(x)+f(w_{x_0,R,\delta}(x))\ge \int_{B_{R}(x_0)}J(x-y)\,v_{x_0,R}(y)\,\mathrm{d}y\ge 0.$$
So let us now assume that $s(x)>\delta^{-1}\tau(x)$, that is,
\be\label{w}
0<w_{x_0,R,\delta}(x)=s(x)-\delta^{-1}\tau(x)\le s(x)=v_{x_0,R}(\mathcal{P}_{x_0,R}(x))<1.
\ee
Let us rewrite the first integral in the right-hand side of~\eqref{eq:subsol2} as
$$\baa{rcl}
\ds\int_{B_{R}(x_0)}J(x-y)\,v_{x_0,R}(y)\,\mathrm{d}y & = & \ds\int_{B_{R}(x_0)}\left[J(x-y)-J(\p{x_0,R}(x)-y)\right]v_{x_0,R}(y)\,\mathrm{d}y\vspace{3pt}\\
& & \ds+\int_{B_{R}(x_0)}J(\p{x_0,R}(x)-y)v_{x_0,R}(y)\,\mathrm{d}y.\eaa$$
Since $v_{x_0,R}$ solves~\eqref{eq:BR} in $\overline{B_R(x_0)}$, since $\p{x_0,R}(x)\in\overline{B_R(x_0)}$ and $s(x)=v_{x_0,R}(\p{x_0,R}(x))$, and since $J\in W^{1,1}(\R^N)$, the above equality yields
\begin{align*}
\int_{B_{R}(x_0)}J(x-y)\,v_{x_0,R}(y)\,\mathrm{d}y &\ge s(x)-f(s(x))- \int_{B_{R}(x_0)}\left|J(x-y)-J(\p{x_0,R}(x)-y)\right|\,\mathrm{d}y.\\
&\ge s(x)-f(s(x))- \int_{\R^N}\left|J(x-y)-J(\p{x_0,R}(x)-y)\right|\,\mathrm{d}y.\\
&\ge s(x)-f(s(x))- \tau(x)\times\int_{\R^N}|\nabla J(z)|\,\mathrm{d}z.
\end{align*}
Combining now the above inequality with~\eqref{eq:subsol2} and $s(x)-\delta^{-1}\tau(x)>0$, and using~\eqref{eq:gamma},~\eqref{defdelta0} and~\eqref{w}, we get
$$\baa{l}
\ds\int_{B_{R+\delta}(x_0)}\!\!\!J(x-y)\,w_{x_0,R,\delta}(y)\,\mathrm{d}y-w_{x_0,R,\delta}(x)+f(w_{x_0,R,\delta}(x))\vspace{3pt}\\
\qquad\qquad\ge g(s(x)) -g(s(x)-\delta^{-1}\tau(x))- \gamma\,\delta_0^{-1}\,\tau(x)\ge (\gamma\,\delta^{-1} -\gamma\,\delta_0^{-1}) \tau(x)\ge0.\eaa$$
This is the desired inequality and the proof of Lemma~\ref{cla:subsol} is thereby complete.
\end{proof}


\SE{The case of convex obstacles: proofs of the main Liouville type results}\label{8}

In this section, we prove our main results. We first consider in Section~\ref{geneker} the case where $J$ is a general kernel satisfying~\eqref{C2}, namely we prove Theorems~\ref{TH:LIOUVILLE} and~\ref{PROP:CONTINUE}. Once this is done, we consider in Section~\ref{comker} kernels having compact support and we prove Theorem~\ref{TH:LIOUVILLE2}. Section~\ref{sec73} is devoted to the proof of a lemma used in the proof of Theorem~\ref{TH:LIOUVILLE2}. Throughout Section~\ref{8}, we always assume that $K$ is a compact convex set and that $f$ and $J$ satisfy the conditions~\eqref{C1},~\eqref{C2} and~\eqref{C3}.


\subsection{General kernels: proofs of Theorems~\ref{TH:LIOUVILLE} and~\ref{PROP:CONTINUE}}\label{geneker}

Let us start our proof of Theorem~\ref{TH:LIOUVILLE} with the following simple observation.

\begin{lemma}\label{CLAIM0}
Let $K\subset\R^N$ be a compact convex set and assume~\eqref{C1} and~\eqref{C2}. Let  $u\in C(\overline{\R^N\setminus K},[0,1])$ satisfy~\eqref{LIM:u}, that is,
\begin{align}
Lu+f(u)\le0\,  & \,\text{ in }\,\overline{\R^N\setminus K}, \label{EQ:Liouville} \vspace{3pt}\\
u(x)\to1\,  & \,\text{ as }\,|x|\to+\infty.  \label{EQ:Liouvillelim}
\end{align}
Then  there exists $\gamma\in(0,1]$ such that $\gamma\leq u\leq 1$ in $\overline{\R^N\setminus K}$.
\end{lemma}

\begin{proof}
We proceed by contradiction. Suppose that the conclusion does not hold. Then, by continuity of $u$ and~\eqref{EQ:Liouvillelim}, there exists a point $x_0\in\overline{\R^N\setminus K}$, such that $u(x_0)=0$. Arguing as in the proof of the strong maximum principle (Lemma~\ref{STRONG:LE23}) or in the proof of
the sweeping principle (Lemma~\ref{TH:SWEEP}), we get that $u=0$ in $\overline{\mathcal{A}(x_0,r_1,r_2)}\cap\overline{\R^N\setminus K}$, where $0\le r_1<r_2$ are given in~\eqref{C2}, and then $u=0$ in $\overline{\mathcal{A}(x_1,r_1,r_2)}\cap\overline{\R^N\setminus K}$ for all $x_1\in\overline{\mathcal{A}(x_0,r_1,r_2)}\cap\overline{\R^N\setminus K}$. Since $K$ is convex, it follows in particular that $u=0$ in $\overline{B_r(x_0)}\cap\overline{\R^N\setminus K}$ for some $r>0$. Finally, the non-empty set $\big\{x\in\overline{\R^N\setminus K};\ u(x)=0\big\}$ is both (obviously) closed and open relatively to the connected set $\overline{\R^N\setminus K}$. Hence $u=0$ in $\overline{\R^N\setminus K}$, contradicting~\eqref{EQ:Liouvillelim}.
\end{proof}

We now turn to the proof of Theorem~\ref{TH:LIOUVILLE}.

\begin{proof}[Proof of Theorem~$\ref{TH:LIOUVILLE}$]
Let $K$, $f$, $J$ and $u$ be as in Theorem~\ref{TH:LIOUVILLE}. Firstly, without loss of generality, one can assume by~\eqref{C1} that $f$ is
extended to a $C^1(\R)$ function satisfying~\eqref{well+2}. Secondly, by~\eqref{LIM:u} and the boundedness of $K$, there exists~$R_0>0$ large enough so that~$K\subset B_{R_0}$ and~$u\ge 1-c_0$ in~$\R^N\setminus B_{R_0}$, where $c_0>0$ is given in~\eqref{well+2}.\par
We proceed the proof by contradiction, and suppose that
\begin{align}
\inf_{\overline{\R^N\setminus K}}\,u<1. \label{hypoko}
\end{align}
{F}rom~\eqref{LIM:u} and~\eqref{hypoko}, together with the continuity of $u$, there exists then $x_0\in\overline{\R^N\setminus K}$ such that
$$u(x_0)=\min_{\overline{\R^N\setminus K}} u\in[0,1).$$
We observe that, by Lemma~\ref{CLAIM0}, one has $u(x_0)>0$. Now, since $K$ is convex, there exists $e\in \partial B_1$ such that  $K\subset H_{e}^c$, where $H_{e}$ is the open affine half-space defined by
$$H_{e}:=x_0+\big\{x\in\R^N;\ x\cdot e> 0\big\}.$$
In light of
assumption~\eqref{C3}, there exists an increasing function $\phi\in C(\R)$ such that
\begin{align*}
\left\{
\begin{array}{r}
J_1\ast \phi-\phi+f(\phi)\geq0\ \text{ in }\R, \vspace{3pt}\\
\phi(-\infty)=0,\ \ \phi(+\infty)=1.
\end{array}
\right.
\end{align*}
Let us also define the function
$$ \varphi_r(x):=\phi_{r,e}(x)=\phi(x\cdot e-r),\ \ x\in\R^N,$$
and the following quantity
$$ r_* := \inf\big\{ r\in \R\,;\ \varphi_r\le u\mbox{ in }\overline{\R^N\setminus K}\big\}.$$
{F}rom Lemmas~\ref{sub-solution} and~\ref{CLAIM0}, we know that $r_*\in [-\infty, r_0]$, where $r_0>0$ is given in Lemma~\ref{sub-solution}.

We claim that in fact
\begin{equation}\label{GOAL}
r_*=-\infty.
\end{equation}
The proof of~\eqref{GOAL} is by contradiction. We assume that~$r_*\in\R$. Then, there exists a sequence~$(\varepsilon_j)_{j\in\N}$ of
positive real numbers such that~$\varphi_{r_*+\varepsilon_j}(x)=\phi(x\cdot e-r_*-\varepsilon_j)\le u(x)$ for all~$x\in\overline{\R^N\setminus K}$ and $\varepsilon_j\to0$ as $j\to+\infty$. Thus passing to the limit as~$j\to+\infty$, we obtain that
$$\varphi_{r_*}(x)\le u(x)\quad{\mbox{ for all }} x\in \overline{\R^N\setminus K}.$$
Let us denote $H$ the open affine half-space
$$H=\big\{x\in \R^N;\ x\cdot e>R_0\big\}.$$
Notice that $\overline{H}\cap K=\emptyset$ and that $u$ is well defined and continuous in $\overline{H}$. 
We also observe that, by construction,
\begin{align}
\sup_{H^c}\varphi_{r_*}<1. \label{bornsup}
\end{align}
Two cases may occur.\par
{\it Case 1:} $\inf_{H^c\setminus K}(u-\varphi_{r_*})>0$. In this situation, thanks to the uniform continuity of $\phi$,
there exists~$\eps>0$ such that $$\inf_{H^c\setminus K}(u-\varphi_{r_*-\eps})>0.$$ Now, we observe that $u$ and $\varphi_{r_*-\eps}$ satisfy
$$\left\{
\begin{array}{rl}
Lu+f(u)\le 0  & \text{in }\overline{H},\vspace{3pt}\\
L\varphi_{r_*-\eps}+f(\varphi_{r_*-\eps})  \ge 0 & \text{in }\overline{H}\ \ \hbox{(by }\eqref{eq15}\hbox{)},\vspace{3pt}\\
u\ge \varphi_{{r_*-\eps}} & \text{in }H^c\setminus K,
\end{array}
\right.$$
together with $u\ge1-c_0$ in $\R^N\!\setminus\! B_{R_0}\supset\overline{H}$ and $\lim_{|x|\to+\infty}u(x)=1$, while $\varphi_{r_*-\eps}\le1$ in $\R^N$. Thus, by the weak maximum principle (Lemma~\ref{WEAK}) and the continuity of $u$ and $\varphi_{r_*-\eps}$ in $\overline{\R^N\setminus K}$, we get that $u\ge \varphi_{r_*-\eps}$ in $\overline{\R^N\setminus K}$. This contradicts the minimality of $r_*$ and therefore Case~1 is ruled out.\par
{\it Case 2:} $\inf_{H^c\setminus K}(u-\varphi_{r_*})=0$. In this situation, by~\eqref{EQ:Liouvillelim} and~\eqref{bornsup}, and by continuity of $u$ and $\varphi_{r_*}$,
there exists a point $\bar x\in\overline{H^c\setminus K}$ such that $u(\bar x)=\varphi_{r_*}(\bar x)$. Note that $\bar x \in \overline{H_e}$, since otherwise $\bar x\in \R^N\setminus\overline{H_e}$, namely $\bar{x}\cdot e<x_0\cdot e$, and
the chain of inequalities
$$u(\bar x)=\varphi_{r_*}(\bar x)<\varphi_{r_*}(x_0)\le u(x_0)=
\min_{\overline{\R^N\setminus K}}u$$
leads to a contradiction. Therefore, we have $\varphi_{r_*}\le u$ in $\overline{\R^N\setminus K}$ with equality at a point $\bar{x}\in\overline{\R^N\setminus K}\cap\overline{H_e}$. Since $K\subset H_e^c$ and $\varphi_{r*}$ and $u$ satisfy respectively
\begin{equation*}
\left\{\baa{rl}
Lu+f(u)\leq 0 & {\mbox{ in }}\overline{H_e},\vspace{3pt}\\
L\varphi_{r_*}+f(\varphi_{r_*})\ge 0 & {\mbox{ in }}\overline{H_e}\ \ \hbox{(by }\eqref{eq15}\hbox{)},\eaa\right.,
\end{equation*}
it follows in particular from the strong maximum principle (Lemma~\ref{STRONG:LE23}) that~$\varphi_{r_*}=u$ in~$\overline{H_e}$. Thus, for any $e^{\perp}\in\partial B_1$ such that $e^\perp\cdot e=0$, one infers from~\eqref{EQ:Liouvillelim} and the definition of $\varphi_{r_*}$ that
$$ 1=\lim_{t\to+\infty} u(x_0+t\,e^{\perp})=\lim_{t\to+\infty} \varphi_{r_*}(x_0+t\,e^{\perp}) =\varphi_{r_*}(x_0)<1.$$
This contradiction rules out Case~2 too.\par
Hence~\eqref{GOAL} holds true and as a consequence  we have that~$\varphi_r\le u$ in $\overline{\R^N\setminus K}$ for any~$r\in\R$. In particular, recalling that $\phi(+\infty)=1$, we get that
$$1> u(x_0)\ge \lim_{r\to-\infty} \varphi_r(x_0) =\lim_{r\to-\infty} \phi(x_0\cdot e-r) =1,$$
a contradiction. Therefore,~\eqref{hypoko} can not hold. In other words, $\inf_{\overline{\R^N\setminus K}}u=1$, i.e. $u=1$ in $\overline{\R^N\setminus K}$. The proof of Theorem~\ref{TH:LIOUVILLE} is thereby complete.
\end{proof}

We observe that, by the same token, we obtain Theorem~\ref{PROP:CONTINUE}.

\begin{proof}[Proof of Theorem~$\ref{PROP:CONTINUE}$]
By Lemma~\ref{LEMMA:INT}, Remark~\ref{RK:CVX} and our assumptions on $f$, we know that $u$ has a (uniformly) continuous representative $u^*\in C(\overline{\R^N\setminus K})$ in its class of equivalence and we can identify $u$ with $u^*$. The desired result now follows as a consequence of Theorem~\ref{TH:LIOUVILLE}.
\end{proof}


\subsection{Compactly supported kernels: proof of Theorem~\ref{TH:LIOUVILLE2}}\label{comker}

In this subsection we prove Theorem~\ref{TH:LIOUVILLE2}. That is, provided some  additional assumptions on $f$ and $J$ are satisfied, we show that the Liouville  result obtained in Theorem~\ref{TH:LIOUVILLE} holds true when the uniform limit of $u$ as $|x|\to +\infty$, namely condition~\eqref{EQ:Liouvillelim}, is replaced by the following weaker condition
\begin{equation}\label{EQ:Sup}
\mathop{{\rm{ess}}\,{\rm{sup}}}_{\R^N\setminus K}\,u =1,
\end{equation}
where $u:\R^N\setminus K\to[0,1]$ is a measurable solution of $Lu+f(u)=0$ a.e. in $\R^N\setminus K$. The condition~\eqref{EQ:Sup} can be rewritten as
\begin{equation}\label{EQ:Sup2}
\sup_{\R^N\setminus K}\,u =1,
\end{equation}
if $u$ is already assumed to be uniformly continuous in $\overline{\R^N\setminus K}$. Note that the extra assumptions~\eqref{CdN-f} made on $f$ (namely $f'<1/2$ in $[0,1]$) actually guarantees that $u$ has a uniformly continuous representative in its class of equivalence, as follows from Lemma~\ref{LEMMA:INT} and Remark~\ref{RK:CVX}. As a consequence, in the proof of Theorem~\ref{TH:LIOUVILLE2} we can assume without loss of generality that $u:\overline{\R^N\setminus K}\to[0,1]$ is uniformly continuous and satisfies~\eqref{EQ:Sup2}. Notice immediately that the same arguments as in the proof of Lemma~\ref{CLAIM0} imply that
\be\label{u>0}
u>0\ \hbox{ in }\overline{\R^N\setminus K}.
\ee
Otherwise $u$ would be identically equal to $0$ in $\overline{\R^N\setminus K}$, contradicting the assumption~\eqref{EQ:Sup2}.

The key-point in the proof of Theorem~\ref{TH:LIOUVILLE2} is the following lemma.

\begin{lemma}\label{lem:lim}
Let $K\subset\R^N$ be a compact set and assume that  $f$ and $J$ satisfy ~\eqref{C1},~\eqref{C2} and~\eqref{C4}. Assume further that $J$ is compactly supported and $J\in L^2(\R^N)$. Let $u:\overline{\R^N\setminus K}\to[0,1]$ be a uniformly continuous solution of~\eqref{LIM:cu9}. Then, $u(x)\to1$ as $|x|\to+\infty$.
\end{lemma}

The proof of Lemma~\ref{lem:lim} is postponed in Section~\ref{sec73}. In this section, we complete the proof of Theorem~\ref{TH:LIOUVILLE2}.

\begin{proof}[Proof of Theorem~$\ref{TH:LIOUVILLE2}$]

{F}rom the previous paragraphs, the function $u$ can be assumed to be uniformly continuous in $\overline{\R^N\setminus K}$ without loss of generality. Then, since the condition~\eqref{C4}, together with~\eqref{C1} and~\eqref{C2}, implies the condition~\eqref{C3}, the assumptions of Theorem~\ref{TH:LIOUVILLE} are all fulfilled, thanks to Lemma~\ref{lem:lim}. Therefore $u=1$ in $\overline{\R^N\setminus K}$, completing the proof of Theorem~\ref{TH:LIOUVILLE2}.
\end{proof}


\subsection{Proof of Lemma~\ref{lem:lim}}\label{sec73}

This section is devoted to the proof of Lemma~\ref{lem:lim}. It is divided into four main steps. To prove Lemma~\ref{lem:lim}, it suffices to show that for any $\eps>0$ small enough there exists $R(\eps)>0$ such that $u\ge 1-\eps$ in $\R^N\setminus B_{R(\eps)}$. To obtain such a lower bound, our strategy relies on the existence of continuous families of continuous sub-solutions $w_\tau$ which satisfy $w_\tau \ge 1-\eps$ in $\overline{B_1(x_\tau)}$ for some $x_\tau\in\R^N$ (these sub-solutions are drawn from Section~\ref{sec63}). Then,
we use the sweeping principle to propagate the lower bound satisfied by the $w_\tau$'s  to a lower bound for $u$.
\vskip 0.3cm

\noindent\emph{Step 1: the solution $u$ is close to $1$ in some
large balls}
\vskip 0.3cm
\noindent{}In this step, we show that, for any $\eps>0$, $\ell>0$, and $R>0$, there exists a point $x^*\in\R^N\setminus K$ such that
\be\label{1-eps}
|x^*|>\ell,\ \ \ B_{R}(x^*)\subset\R^N\setminus K,\ \hbox{ and }\ u\ge1-\eps\hbox{ in }\overline{B_{R}(x^*)}.
\ee

To do so, notice first that, from~\eqref{LIM:cu9} and the continuity of $u$ in $\overline{\R^N\setminus K}$, two situations may occur: namely, either there exists a sequence $(x_n)_{n\in\N}\subset\R^N\setminus K$ such that
\be\label{xn}
\lim_{n\to+\infty}|x_n|=+\infty~~~~\text{ and }~~~~\lim_{n\to+\infty}u(x_n)=1,
\ee
or there exists a point $\bar x \in \overline{\R^N\setminus K}$ such that $u(\bar x)=1$. In the latter case, since $f(u(\bar x))=f(1)=0$, we get,
as in the proof of Lemma~\ref{CLAIM0}, that $u=1$ in $\overline{\R^N\setminus K}$: the claim~\eqref{1-eps} is therefore trivial in this case.

Thus, it suffices to treat the former case~\eqref{xn} only. Consider the functions $u_n$ defined in $\overline{\R^N\setminus K}-x_n$ by
$$u_n(x)=u(x+x_n).$$
Since $u$ is uniformly continuous in $\overline{\R^N\setminus K}$ and since $K$ is compact and $\lim_{n\to+\infty}|x_n|=+\infty$, it follows that, for every $r>0$, the functions $u_n$'s, ranging in $[0,1]$, are defined in $\overline{B_r}$ for all $n$ large enough and are uniformly equicontinuous in $\overline{B_r}$. {F}rom Arzela-Ascoli theorem and the diagonal extraction process,
there exists a continuous function $u_{\infty}:\R^N\to[0,1]$ such that, up to extraction of a subsequence, $u_n\to u_{\infty}$ locally uniformly in $\R^N$ as $n\to+\infty$. Furthermore, $u_{\infty}(0)=1$ by~\eqref{xn}. On the other hand, the functions $u_n$'s satisfy
$$\int_{(\R^N\setminus K)-x_n}J(x-y)\,u_n(y)\,\mathrm{d}y-\Big(\int_{(\R^N\setminus K)-x_n}J(x-y)\,\mathrm{d}y\Big)\,u_n(x)+f(u_n(x))=0$$
for all $x\in\overline{\R^N\setminus K}-x_n$. Lebesgue's dominated convergence theorem implies that
$$J*u_{\infty}-u_{\infty}+f(u_\infty)=0\ \hbox{ in }\R^N.$$
Since $f(u_\infty(0))=f(1)=0$ and $u_\infty\le1$ in $\R^N$, we get as in the proof of Lemma~\ref{CLAIM0} that $u_\infty=1$ in $\R^N$. In particular, for any fixed $\eps>0$, $\ell>0$, and $R>0$, it follows that, for every $n\in\N$ large enough, there holds $|x_n|>\ell$, $B_R(x_n)\subset\R^N\setminus K$ and $u_n\ge1-\epsilon$ in $\overline{B_R}$, that is, $u\ge1-\epsilon$ in $\overline{B_R}(x_n)$. In other words, the claim~\eqref{1-eps} holds with $x^*=x_n$ and $n$ large enough.
\vskip 0.3cm

\noindent \emph{Step 2: a sub-solution in a ball}
\vskip 0.3cm
\noindent{}Fix $\eps>0$ small enough so that $f'<0$ in $[1-\eps,1]$, and let us now establish a lower bound for~$u$ in a ball far away from $K$, by using a sub-solution drawn from Section~\ref{sec63}. We recall here that $R_J>0$ is such that $\mathrm{supp}(J)\subset B_{R_J}$.\par
We first claim that there exist $x^*\in\R^N$, $0<R_J\le R_{K}\le R$ and a function $w\in C(\R^N,[0,1))$ such that
\be\label{bchv-eq-subsol-3}\left\{\baa{l}
B_{R+1}(x^*)\ \subset\ \R^N\!\setminus\!B_{R_K}\ \subset\ \R^N\!\setminus\!K,\ \ u\ge1-\eps\text{ in }\overline{B_{R+1}(x^*)},\vspace{3pt}\\
\oplb{w}{B_{R+1}(x^*)}\!-\!w\!+\!f(w)\ge0\hbox{ in }\R^N,\ \ w\!\ge\!1\!-\!\eps\hbox{ in }\overline{B_{1}(x^*)},\ \ w\!=\!0\text{ in }\R^N\!\setminus\!B_{R+1}(x^*).\eaa\right.
\end{equation}
To show this claim, let $R_{K}\ge\max\{1,R_J\}$ be such that $K\subset B_{R_K}$. Then choose $R\ge\max\{R_{K},d_0\}\ge1$ ($d_0>0$ is given as in Lemmas~\ref{existence-BR} and~\ref{bchv-lem-exis-msol}) such that the maximal solution $v_{0,R}\in C(\overline{B_{R}},(0,1))$ to problem~\eqref{eq:BR} in $\overline{B_{R}}$ satisfies
\be\label{v0R1}
v_{0,R}\ge1-\eps\ \hbox{ in }\overline{B_1}.
\ee
Note that such a real number $R$ exists according to Lemmas~\ref{bchv-lem-exis-msol} and~\ref{prop:limVR}. On the one hand, as far as the first line in~\eqref{bchv-eq-subsol-3} is concerned,
formula~\eqref{1-eps}, applied here
with $\ell=R+1+R_K>0$ and $R+1>0$ in place of $R$, yields the existence of $x^*\in\R^N$ such that
\be\label{defx*}
|x^*|>R+1+R_K
\ee
(hence, $B_{R+1}(x^*)\subset\R^N\!\setminus\!B_{R_K}\subset\R^N\!\setminus\!K$) and
\be\label{u1eps}
u\ge1-\eps\ \hbox{ in }\overline{B_{R+1}(x^*)}.
\ee
Thanks to~\eqref{v0R1} and part~(ii) of Lemma~\ref{prop:VR}, the maximal solution $v_{x^*,R}\in C(\overline{B_{R}(x^*)},(0,1))$ to problem~\eqref{eq:BR} in $\overline{B_{R}(x^*)}$
satisfies $v_{x^*,R}\ge1-\eps$ in $\overline{B_1(x^*)}$. On the other hand, as far as the second line in~\eqref{bchv-eq-subsol-3} is concerned, Lemma~\ref{cla:subsol} provides the existence of a function $w\in C(\R^N,[0,1))$ such that
$$\oplb{w}{B_{R+1}(x^*)}-w+f(w)\ge0\hbox{ in }\R^N,\ \ w=v_{x^*,R}\hbox{ in }\overline{B_{R}(x^*)}\supset\overline{B_1(x^*)}$$
and $w=0$ in $\R^N\setminus B_{R+1}(x^*)$. As a consequence, $x^*,R_{K},R$ and $w$ fulfill~\eqref{bchv-eq-subsol-3}.\par
We then claim that
\be\label{wleu}
w\le u\ \hbox{ in }\overline{\R^N\setminus K}.
\ee
Since $w=0$ in $\R^N\setminus B_{R+1}(x^*)$ and $u\ge0$ in $\overline{\R^N\setminus K}$, we only need to show that $w\le u$ in $\overline{B_{R+1}(x^*)}\ (\subset\overline{\R^N\setminus K})$. Denote
$$z:=w-u$$
in $\overline{B_{R+1}(x^*)}$ and assume that
$$\max_{\overline{B_{R+1}(x^*)}}z=z(\bar x)>0$$ for some $\bar x\in\overline{B_{R+1}(x^*)}$. Since $\overline{B_{R+1}(x^*)}\subset\overline{\R^N\setminus K}$ and $u$ and $J$ are non-negative with $J$ having a unit mass in $L^1(\R^N)$, it follows from the equation $Lu+f(u)=0$ satisfied by $u$ in $\overline{\R^N\setminus K}$ that
$$\mathcal{L}_{B_{R+1}(x^*)}[u](\bar{x})-u(\bar x)+f(u(\bar x))\le 0.$$
Together with the first inequality of the second line of~\eqref{bchv-eq-subsol-3} applied at $\bar x$, we get that
\be\label{barx}
\mathcal{L}_{B_{R+1}(x^*)}[z](\bar{x})-z(\bar x)+f(w(\bar x))-f(u(\bar x))\ge 0.
\ee
Since $z\le z(\bar x)$ in $\overline{B_{R+1}(x^*)}$, one has $\mathcal{L}_{B_{R+1}(x^*)}[z](\bar{x})-z(\bar x)\le0$. Furthermore, remembering~\eqref{u1eps} and the choice of $\epsilon$, there holds $1-\epsilon\le u(\bar x)=w(\bar x)-z(\bar x)<w(\bar x)<1$ and $f'<0$ in $[1-\epsilon,1]$, hence $f(w(\bar x))-f(u(\bar x))<0$. This contradicts~\eqref{barx}. Therefore, $\max_{\overline{B_{R+1}(x^*)}}z\le0$, that is, $w\le u$ in $\overline{B_{R+1}(x^*)}$ and then in $\overline{\R^N\setminus K}$.
\vskip 0.3cm

\noindent \emph{Step 3: a lower bound in annuli with large inner radii}
\vskip 0.3cm
\noindent{}Let us now construct some families of sub-solutions and exploit
the sweeping principle (Lemma~\ref{TH:SWEEP}) to get a lower bound of $u$ in some annuli. To do so, let $x^*\in\R^N$, $0<R_J\le R_{K}\le R$ and $w\in C(\R^N,[0,1))$ be as in~\eqref{bchv-eq-subsol-3}. Consider any orthonormal basis $(e_1,\cdots,e_N)$ of $\R^N$ and, for $\tau\in[0,2\pi]$, let $\mathcal{R}_{\tau}$ be the rotation of angle $\tau$ in the plane spanned by $(e_1,e_2)$ (that is, $\mathcal{R}_\tau e_1=(\cos\tau)e_1+(\sin\tau)e_2$ and $\mathcal{R}_\tau e_2=-(\sin\tau)e_1+(\cos\tau)e_2$) and leaving invariant the vectors $e_3,\cdots,e_N$. We set
$$A:=\mathcal{A}(|x^*|-R-1,|x^*|+R+1)=B_{|x^*|+R+1}\setminus\overline{B_{|x^*|-R-1}}.$$
{F}rom~\eqref{bchv-eq-subsol-3}, note that $\overline{A}\,\subset\,\R^N\!\setminus\!B_{R_{K}}\,\subset\,\R^N\!\setminus\!K$ (hence, $\overline{A}\cap K=\emptyset$). Now for each $\tau\in [0,2\pi]$ and $x\in\R^N$, we set
$$w_{\tau}(x):= w(\mathcal{R}_\tau x).$$
Thanks to the rotational invariance of $J$ and $A$, and since $B_{R+1}(x^*)\subset A$ and both $J$ and $w$ are non-negative, it follows from~\eqref{bchv-eq-subsol-3} that each function $w_\tau$ satisfies
$$\oplb{w_\tau}{A}-w_\tau +f(w_\tau)\ge0 \quad{\mbox{in }}\R^N.$$
On the other hand, it follows from~\eqref{LIM:cu9} that the function $u$ obeys
$$\oplb{u}{A}(x)\!-\!u(x)\!+\!f(u(x))=-\int_{\R^N\setminus(K\cup A)}\!\!\!J(x\!-\!y)\,u(y)\,\mathrm{d}y - u(x)\left(1\!-\!\int_{\R^N\setminus K}\!\!J(x\!-\!y)\,\mathrm{d}y\right)\le 0$$
for all $x\in\overline{\R^N\setminus K}$ and therefore for all $x\in\overline{A}$. In addition, thanks to positivity of $u$ in $\overline{\R^N\setminus K}$ (remember~\eqref{u>0}) and the fact that $J>0$ a.e. in $\mathcal{A}(r_1,r_2)$ with $0\le r_1<r_2\le R_J\le R_K\le R$ (remember~\eqref{C2} and $\mathrm{supp}(J)\subset B_{R_J}$), one infers that
$$ \int_{\R^N\setminus(K\cup A)}\!\!J(x\!-\!y)\,u(y)\,\mathrm{d}y>0\ \hbox{ for all }x\in A':=\mathcal{A}(|x^*|\!+\!R\!+\!1\!-\!r_2,|x^*|\!+\!R\!+\!1)\ (\subset A),$$
hence
$$\oplb{u}{A}(x)-u(x) +f(u(x))<0 \ \text{ for all }x \in A'.$$
Since $w\le u$ in $\overline{A}\,(\subset\overline{\R^N\setminus K})$ by~\eqref{wleu} and $r_2\le R_K\le|x^*|-R-1$ by~\eqref{defx*}, it follows
from the sweeping principle (Lemma~\ref{TH:SWEEP}) applied to $u$, to the family $(w_\tau)_{\tau\in[0,2\pi]}$ and to
$$(s_1,s_2,s_3,s_4)=(|x^*|-R-1,|x^*|+R+1-r_2,|x^*|+R+1,|x^*|+R+1),$$
that
\be\label{ineqwtau}
w_\tau\le u\ \hbox{ in }\overline{A}\ \hbox{ for every }\tau\in[0,2\pi].
\ee
Notice also (even if the following inequalities will not explicitly
be used in the next step) that, since $w\ge1-\epsilon$ in $\overline{B_1(x^*)}$
by~\eqref{bchv-eq-subsol-3}, the family
of estimates in~\eqref{ineqwtau} implies in particular that $u\ge1-\eps$ in $\bigcup_{\tau \in [0,2\pi]}\overline{B_1(\mathcal{R}_\tau^{-1}x^*)}$. Since
the previous arguments are
independent of the choice of the orthonormal basis $(e_1,\ldots,e_N)$, we also get that $u\ge1-\eps$ in $\overline{\a(|x^*|-1,|x^*|+1)}$.
\vskip 0.3cm

\noindent \emph{Step 4: conclusion}
\vskip 0.3cm
\noindent{}Let us now finish our argument. To complete the proof of Lemma~\ref{lem:lim}, we will again construct an adequate family of sub-solutions and use
the sweeping principle to push further the estimates obtained in the previous step. To do so, pick some $\rho>0$ and consider the domain
$$A_{\rho}:=\mathcal{A}(|x^*|-R-1,|x^*|+R+1+\rho),$$
where $R>0$ is defined in Steps~2 and~3. {F}rom~\eqref{bchv-eq-subsol-3}, we
note that $\overline{A_\rho}\,\subset\,\R^N\!\setminus\!B_{R_{K}}\,\subset\,\R^N\!\setminus\!K$ (hence, $\overline{A_\rho}\cap K=\emptyset$). Next, consider any rotation $\mathcal{R}$ of $\R^N$, let $e:=x^*/|x^*|\in \partial B_1$ and, for each $\sigma\in [0,\rho]$ and $x\in\R^N$, denote
$$W_{\sigma}(x):= w(\mathcal{R}\,x-\sigma\,e).$$
As in the previous step, from the rotational invariance of $J$ and $A_\rho$, and since $B_{R+1}(x^*+\sigma e)\subset A_\rho$ for every $\sigma\in[0,\rho]$ and both $J$ and $w$ are non-negative, it follows from~\eqref{bchv-eq-subsol-3} that each function $W_\sigma$ satisfies
$$\oplb{W_\sigma}{A_\rho}-W_\sigma +f(W_\sigma)\ge0\ \hbox{ in }\R^N.$$
Similarly, it follows from~\eqref{LIM:cu9} that the function $u$ obeys
$$\oplb{u}{A_\rho}-u+f(u)\le0\ \hbox{ in }\overline{\R^N\setminus K}$$
(and therefore in $\overline{A_\rho}$), while
$$\oplb{u}{A_\rho}-u+f(u)<0 \ \text{ in }\mathcal{A}(|x^*|+R+1+\rho-r_2,|x^*|+R+1+\rho)\ (\subset A_\rho).$$
{F}rom the inequality~\eqref{ineqwtau} of the previous step (which holds for every $\tau\in[0,2\pi]$ and for every orthonormal basis $(e_1,\cdots,e_N)$), we have $W_0\le u$ in $\overline{A}$ and then in $\overline{\R^N\setminus K}$ (since $W_0=0$ in $\R^N\setminus A$ and $u>0$ in $\overline{\R^N\setminus K}$. As a consequence, $W_0\le u$ in $\overline{A_\rho}$. Finally, it follows
from the sweeping principle (Lemma~\ref{TH:SWEEP}) applied to $u$, to the family $(W_\sigma)_{\sigma\in[0,\rho]}$ and to
$$(s_1,s_2,s_3,s_4)=(|x^*|-R-1,|x^*|+R+1+\rho-r_2,|x^*|+R+1+\rho,|x^*|+R+1+\rho),$$
that $W_\sigma\le u$ in $\overline{A_\rho}$ for every $\sigma\in[0,\rho]$. Since $w\ge1-\epsilon$ in $\overline{B_1(x^*)}$ by~\eqref{bchv-eq-subsol-3}, we obtain in particular that
$$u\ge1-\eps\ \ \hbox{ in }\bigcup_{\sigma \in [0,\rho]}\overline{B_1(\mathcal{R}^{-1}(x^*+\sigma e))}.$$\par
The previous arguments being independent of the choice of $\rho>0$ and the rotation $\mathcal{R}$ of $\R^N$, we conclude that
$$u(x)\ge1-\eps\ \ \hbox{ for all }|x|\ge |x^*|-1.$$
Since $\eps>0$ can be arbitrarily small, the proof of Lemma~\ref{lem:lim} is thereby complete.


\SE{The case of small perturbations of convex obstacles}\label{9}

In this section, we explore further the validity of the Liouville Theorem~\ref{TH:LIOUVILLE} and we prove  Theorem~\ref{TH:PERTURB}, a kind of stability result for the Liouville property. In the spirit of the results of Bouhours~\cite{Bouhours}, we show that the Liouville property obtained in  Theorem~\ref{TH:LIOUVILLE} still holds true for small perturbations of convex obstacles, provided some additional assumptions are made on $f$ and $J$. To do so, we adapt to our problem the arguments developed in~\cite{Bouhours} and, in particular, we will rely on the following

\begin{lemma}\label{PROP:UNIFCV}
Assume all hypotheses of Theorem~$\ref{TH:PERTURB}$. Then, for every $\delta\in(0,1)$, there exists a real number $R_\delta>0$ such that, for any $\eps\in(0,1]$ and any measurable solution $u_\eps:\R^N\setminus K_\eps\to[0,1]$ of~\eqref{eq:ueps}, there holds $u_\eps(x)\geq 1-\delta$ for a.e. $|x|\geq R_\delta$.
\end{lemma}

Before proving Lemma~\ref{PROP:UNIFCV}, let us first establish a preliminary ``rough" Liouville-type result, namely Proposition~\ref{LargeGamma}.

\begin{proof}[Proof of Proposition~$\ref{LargeGamma}$]
We recall that $f\in C^1([0,1])$, that $J$ is assumed to satisfy~\eqref{C2}, that $K$ is a compact set such that $\R^N\setminus K$ is connected, and that $u:\overline{\R^N\setminus K}\to[\theta,1]$ is a continuous solution of~\eqref{LIM:u0} such that $f\ge0$ on $[\theta,1]$. Let us set
$$m=\inf_{\overline{\R\setminus K}}u\ \in[\theta,1].$$
Suppose, by contradiction, that $m<1$. Let $(x_n)_{n\in\N}\subset\overline{\R^N\setminus K}$ be a sequence such that $u(x_n)\to m$ as $n \to +\infty$. Since  $u(x)\to 1$ as $|x|\to +\infty$, the sequence $(x_n)_{n\in \N}$ is bounded and, up to extraction of a subsequence, we may assume that it converges to some $\bar x \in \overline{\R^N\setminus K}$. Evaluating the equation satisfied by $u$ at $x_n$, we obtain
$$ \int_{\R^N\setminus K}J(x_n-y)\big(u(y)-u(x_n)\big)\,\D y +f\big(u(x_n)\big)=0.$$
By assumption, $f(u(x))\ge 0$ for all $x\in\overline{\R^N\setminus K}$ and therefore
$$ \int_{\R^N\setminus K}J(x_n-y)\big(u(y)-u(x_n)\big)\,\D y \le 0.$$
Since $J \in L^1(\R^N)$ passing to the limit in the above inequality results in
$$0\le \int_{\R^N\setminus K}J(\bar x-y)\big(u(y)-m\big)\,\D y \le 0.$$
Thus, arguing as in Section~\ref{4} and using~\eqref{C2} and the connectedness of $\R^N\setminus K$, we obtain that $u=m\ (<1)$ in $\overline{\R^N\setminus K}$. Since $u(x)\to 1$ as $|x|\to +\infty$, we get a contradiction. The proof of Proposition~\ref{LargeGamma} is thereby complete.
\end{proof}

Let us now turn our attention to the proof of Lemma~\ref{PROP:UNIFCV}.

\begin{proof}[Proof of Lemma~$\ref{PROP:UNIFCV}$]
First of all, in virtue of Lemma~\ref{LEMMA:INT}, we know that, for every $\eps\in(0,1]$, every measurable solution $u_\eps:\R^N\setminus K_\eps\to[0,1]$ of~\eqref{eq:ueps} possesses a H\"older continuous representative $u_\eps^*\in C^{0,\alpha}(\overline{\R^N\setminus K_\eps})$. Consequently, we are allowed to identify $u_\eps$ with $u_\eps^*$. For simplicity, we omit the superscript $*$ and write simply $u_\eps$ instead of $u_\eps^*$.\par
Let us then continuously extend $f$ by $f'(0)s$ for $s\le 0$ and by $f'(1)(s-1)$ for $s\ge 1$  and  still denote $f$ this extension. 
We also observe that, since $(K_\eps)_{0<\eps\le1}$ is a family of (at least) $C^{0,\alpha}$ deformations of $K$ in the sense of Definition~\ref{DEF:KEPS},
there exists a real number $R_0>0$ such that
\be\label{defR0}
K_\eps\subset B_{R_0} \quad{\mbox{ for all }}0<\eps\leq1.
\ee\par
Notice now that it is sufficient to show the conclusion of Lemma~\ref{PROP:UNIFCV} for $\delta>0$ small enough. For any $\delta>0$ small enough, we are going to consider an auxiliary problem whose solutions will provide an appropriate lower bound for $u_\eps$, allowing us to prove the desired uniform convergence as $|x|\to+\infty$. To this end, for $\delta\in(0,1)$, denote
$$f_\delta(s):=f(s)-f(1-\delta/2)\hbox{ for }s\in\R,\ \hbox{ and }\ s_\delta:=\frac{f(1-\delta/2)}{f'(0)}.$$
It is immediate to check that there exists $\delta_1\in(0,1)$ such that, for every $\delta\in(0,\delta_1)$, one has $s_\delta<0<1-\delta/2<1$ and
$$\left\{
\begin{array}{l}
f_\delta\leq f\text{ in }\R,\ \ f_\delta'=f'<1/2\text{ in }\R,\vspace{3pt}\\
f_\delta(s_\delta)=0,\ \ f_\delta'(s_\delta)<0,\ \ f_\delta(1-\delta/2)=0,\ \ f_\delta'(1-\delta/2)<0,\ \ \displaystyle\int_{s_\delta}^{1-\delta/2}f_\delta(r)\,\mathrm{d}r>0,\vspace{3pt}\\
f_\delta\text{ vanishes only once in }(s_\delta,1-\delta/2).\end{array}
\right.$$
Using the results obtained in~\cite{Bates,Chen,Coville,Yagisita}, we know that, for every $\delta\in(0,\delta_1)$, there exists a continuous function $\phi_\delta:\R\to(s_\delta,1-\delta/2)$ satisfying
$$\left\{
\begin{array}{l}
L_{\R}\phi_\delta+f_\delta(\phi_\delta)=J_1*\phi_\delta-\phi_\delta+f_\delta(\phi_\delta)\geq0\ \text{ in }\R,\vspace{3pt}\\
\phi_\delta\hbox{ is increasing in }\R,\vspace{3pt}\\
\phi_\delta(-\infty)=s_\delta,\ \ \phi_\delta(0)=0,\ \ \phi_\delta(+\infty)=1-\delta/2.
\end{array}
\right.$$\par
Fix in the sequel any $\delta\in(0,\delta_1)$, any $\epsilon \in (0,1]$ and any (H\"older-continuous) function $u_\epsilon:\overline{\R^N\setminus K_\eps}\to[0,1]$ solving~\eqref{eq:ueps}. For $A>0$, we let $\Phi_{\delta,A}$ be the function defined in $\R^N$ by
$$ \Phi_{\delta,A}(x):=\phi_\delta(|x|-A). $$
We observe that, by construction, we have
\be\label{Phidelta}
\Phi_{\delta,R_0}(x)\le0\le u_{\epsilon} \quad \text{ for all } x \in\overline{B_{R_0}\setminus K_\eps}.
\ee
Our aim is to extend the above relation to all $x \in \R^N\setminus\overline{B_{R_0}}$. Since $u_\epsilon(x) \to 1$ as  $|x| \to +\infty$, there exists $R_\epsilon>R_0$ such that
$u_{\epsilon}(x)\ge\max(1-c_0,1-\delta/2)$ for all $|x|\ge R_\epsilon$ where  $c_0>0$ is such that $f'<0$ in $[1-c_0,+\infty)$. Then, reasoning as in Lemma~\ref{sub-solution} (or using directly that $\Phi_{\delta,A}\to s_\delta<0$ as $A\to+\infty$ locally uniformly in $\R^N$ and $\Phi_{\delta,A}<1-\delta/2<1$ in $\R^N$), we obtain that, for some $A_\eps>0$,
$$ \Phi_{\delta,A_{\epsilon}}\le u_\epsilon\ \text{ in }\overline{\R^N\setminus K_\eps}.$$
Consequently, it makes sense to define
$$ A^*:=\inf\big\{A\in\R\,;\ \Phi_{\delta,A}\le u_\eps\hbox{ in }\overline{\R^N\setminus K_\epsilon}\big\}\ \le A_\eps.$$\par
We claim that
\begin{equation}\label{Claim:A0}
A^* \le R_0
\end{equation}
We argue by contradiction and assume that $A^*>R_0$. {F}rom the definition of $A^*$ and the continuity of $\phi_\delta$, we have
\be\label{PhideltaA*}
\Phi_{\delta,A^*}\le u_{\epsilon}\ \hbox{ in }\overline{\R^N\setminus K_\epsilon}.
\ee
If $\min_{\overline{B_{R_\eps}\setminus K_\eps}}(u_\eps-\Phi_{\delta,A^*})>0$, then from the uniform continuity of $\phi_\delta$,
there exists $\tau>0$ small enough such that $\Phi_{\delta,A^*-\tau}\le u_{\epsilon}$ in $\overline{B_{R_\eps}\setminus K_\epsilon}$. On the other hand, $\Phi_{\delta,A^*-\tau}<1-\delta/2\le u_\eps$ in $\R^N\setminus B_{R_\eps}$. Hence, $\Phi_{\delta,A^*-\tau}\le u_{\epsilon}$ in $\overline{\R^N\setminus K_\epsilon}$, a contradiction with the definition of $A^*$. Therefore, $\min_{\overline{B_{R_\eps}\setminus K_\eps}}(u_\eps-\Phi_{\delta,A^*})=0$. Since $u_\epsilon$ and $\Phi_{\delta,A^*}$ are continuous, there exists $x_0 \in \overline{B_{R_\eps}\setminus K_\eps}$ such that
$$\Phi_{\delta,A^*}( x_0)= u_{\epsilon}(x_0).$$
Since $A^*>R_0$ by assumption, it follows from~\eqref{Phidelta} and the strict monotonicity of $\Phi_{\delta,A}$ with respect to $A$ that  $x_0 \in \overline{B_{R_\eps}\setminus B_{R_0}}$. Let us set $e_0=x_0/|x_0|$ and define the open affine half-space
$$H:=\big\{x\in\R^N\,;\ x\cdot e_0 >R_0\big\}\ (\subset\R^N\setminus K_\eps).$$
{F}rom~\eqref{PhideltaA*} and the definition of $\Phi_{\delta,A^*}$, we have
$$u_\eps(x)\geq \varphi(x):=\phi_{\delta}(x\cdot e_0 -A^*)\quad{\mbox{for all }}x\in\overline{\R^N\setminus K_\eps}.$$
Reasoning as in Lemma~\ref{sub-solution} and recalling the assumptions on $f_\delta$, we have that
$$\left\{
\begin{array}{r l}
L_\eps u_\eps+f(u_\eps)=0 & \text{in }\overline{H},\vspace{3pt}\\
L_\eps \varphi+f(\varphi)\geq0 & \text{in }\overline{H}\ \ \hbox{(as in }\eqref{eq15}\hbox{)},\vspace{3pt}\\
u_\eps\geq \varphi & \text{in }\overline{\R^N\setminus K_\eps},\vspace{3pt}\\
u_\eps(x_0)=\varphi(x_0) & \text{with }x_0\in\overline{H}.
\end{array}
\right.$$
Applying the strong maximum principle (Lemma~\ref{STRONG:LE23}) we obtain in particular that $u_\eps=\varphi$ in $\overline{H}$. This is impossible since $u_\eps(x)\to1$ as $|x|\to+\infty$, while $\varphi<1-\delta/2<1$ in $\R^N$. As a consequence, the claim~\eqref{Claim:A0} holds true.\par
{F}rom~\eqref{Claim:A0} and the monotonicity of $\Phi_{\delta,A}$ with respect to $A$, we then deduce that
$$\Phi_{\delta,R_0}\le\Phi_{\delta,A^*}\le u_\eps\ \hbox{ in }\overline{\R^N\setminus K_\eps}.$$
Since $\epsilon\in(0,1]$ and $u_\eps:\overline{\R^N\setminus K_\eps}\to[0,1]$ solving~\eqref{eq:ueps} were arbitrary, since $R_0>0$ verifying~\eqref{defR0} was independent of $\eps$, and since $\phi_\delta(+\infty)=1-\delta/2>1-\delta$, the desired conclusion follows.
\end{proof}

We are now ready to prove Theorem~\ref{TH:PERTURB}.

\begin{proof}[Proof of Theorem~$\ref{TH:PERTURB}$]
First of all, as in the proof of Lemma~\ref{PROP:UNIFCV}, it follows from Lemma~\ref{LEMMA:INT} that, for every $\eps\in(0,1]$, every measurable solution $u_\eps:\R^N\setminus K_\eps\to[0,1]$ of~\eqref{eq:ueps} can be identified with its H\"older continuous $C^{0,\alpha}(\overline{\R^N\setminus K_\eps})$ representative. Furthermore, Lemma~\ref{LEMMA:INT} yields
$$[u_\eps]_{C^{0,\alpha}(\overline{\R^N\setminus K_\eps})}\leq A:=\frac{2[J]_{B_{1,\infty}^\alpha(\R^N)}}{\displaystyle\inf_{0<\eta\leq 1}\inf_{x\in\R^N\setminus K_\eta}\|J(x-\cdot)\|_{L^1(\R^N\setminus K_\eta)}-\max_{[0,1]}f'}.$$
Note that $A$ is independent of $\eps$. In particular, for every $\eps_*\in(0,1]$ and every $R\ge R_0$, where $R_0>0$ is chosen as in~\eqref{defR0}, the family $(u_\eps)_{0<\eps\leq\eps_*}$ is uniformly bounded in $C^{0,\alpha}(\overline{B_R\setminus K_{\eps_*}})$. Recalling that $K_\eps\to K$ as $\eps\to0^+$ in the $C^{0,\alpha}$ sense,
there exists a sequence $(\eps_j)_{j\in\N}\in(0,1]$ converging to $0^+$ and a function $u_0\in C^{0,\alpha}(\overline{\R^N\setminus K})$ such that, for all $R\ge R_0$ and $\beta\in(0,\alpha)$,
\be\label{uepsk}
\|u_{\eps_j}-u_0\|_{C^{0,\beta}(\overline{B_R\setminus K_{\eps_j}})}\to0 \quad{\mbox{as }}j\to+\infty.
\ee
Notice that $0\le u_0\le 1$ in $\overline{\R^N\setminus K}$. By Lemma~\ref{PROP:UNIFCV} we know that $u_\eps(x)\to1$ uniformly in $\eps>0$ as $|x|\to+\infty$. Consequently,
\begin{align}
u_0(x)\to1 \quad{\mbox{ as }}|x|\to+\infty. \label{u0infty}
\end{align}\par
Now, we claim that
\begin{align}
Lu_0(x)+f(u_0(x))=0 \quad{\mbox{ in }}\overline{\R^N\setminus K}, \label{EQ:u0}
\end{align}
where $L$ is given by~\eqref{DEF:L}. This can be seen as follows. First, fix $x$ in the open set $\R^N\setminus K$ and an integer $j_0$ large enough such that $x\in\R^N\setminus K_{\eps_j}$ for all $j\ge j_0$. Notice that $f(u_{\eps_j}(x))\to f(u_0(x))$ as $j\to+\infty$ since $f$ is continuous. Next, for all $j\geq j_0$ we have
$$\begin{array}{rcl}
L_{\eps_j} u_{\eps_j}(x)-Lu_0(x) & = & \displaystyle\int_{\R^N\setminus K_{\eps_j}}J(x-y)\big[(u_{\eps_j}-u_0)(y)-(u_{\eps_j}-u_0)(x)\big]\mathrm{d}y\vspace{3pt}\\
& & \displaystyle-\int_{K_{\eps_j}\setminus K}J(x-y)\big(u_0(y)-u_0(x)\big)\,\mathrm{d}y.\end{array}$$
For every $R\geq R_0$ and $j\ge j_0$, there holds
$$\begin{array}{rcl}
|L_{\eps_j} u_{\eps_j}(x)-Lu_0(x)| & \leq & \displaystyle2\int_{K_{\eps_j}\setminus K}J(x-y)\,\mathrm{d}y+2\int_{\R^N\setminus B_R} J(x-y)\,\mathrm{d}y\vspace{3pt}\\
& & \displaystyle+\|u_{\eps_j}-u_0\|_{L^{\infty}(B_R\setminus K_{\eps_j})}+|u_{\eps_j}(x)-u_0(x)|.\end{array}$$
Since $K_{\eps_j}\to K$ in the $C^{0,\alpha}$ sense and $J\in L^1(\R^N)$, we have in particular that the first term in the right-hand side converges to $0$ as $j\to+\infty$. Recalling~\eqref{uepsk} and letting first $j\to+\infty$ and then $R\to+\infty$, we find that
$$L_{\eps_j} u_{\eps_j}(x)-Lu_0(x)\to 0\ \hbox{ as }j\to+\infty.$$
Therefore,~\eqref{EQ:u0} holds for all $x\in\R^N\setminus K$ and finally for all $x\in\overline{\R^N\setminus K}$ by continuity and boundedness of $u_0$ in $\overline{\R^N\setminus K}$.\par
Remember now that $u_0\in C(\overline{\R^N\setminus K},[0,1])$. By~\eqref{u0infty},~\eqref{EQ:u0} and Theorem~\ref{TH:LIOUVILLE}, we infer that $u_0=1$ in $\overline{\R^N\setminus K}$. This also shows that the limit of the functions $u_{\epsilon_j}$ is unique and, hence, $u_\eps\to1$ as $\eps\to0^+$ in the sense of~\eqref{uepsk}, not only along a subsequence.\par
We conclude by contradiction. Suppose then that there exists countably infinitely many numbers in $(0,1]$, which we label in decreasing order as $(\eps_j)_{j\in\N}$, such that $ \eps_j\to 0^+$ as $j\to+\infty$ and
\begin{align}
\forall\,j\in\N,\ \ \exists\,x_j\in\overline{\R^N\setminus K_{\eps_j}},\ \ u_{\eps_j}(x_j)=\min_{\overline{\R^N\setminus K_{\eps_j}}}\,u_{\eps_j}<1.  \label{H:ABS:2}
\end{align}
Note that this makes sense since, without loss of generality, we have identified the functions $u_{\eps_j}$ with their continuous representatives in $\overline{\R^N\setminus K_{\eps_j}}$. 
We observe that~\eqref{C4},~\eqref{H:ABS:2} and Proposition~\ref{LargeGamma} yield that
$$ u_{\eps_j}(x_j)<\theta \quad{\mbox{for all }}j\in\N. $$
Now, since the functions $u_{\eps_j}$ converge uniformly to $1$ as $|x|\to+\infty$ (by Lemma~\ref{PROP:UNIFCV}), the sequence $(x_j)_{j\in\N}$ is bounded. Hence, up to extraction of a subsequence, we may assume that $x_j\to\bar{x}$ as $j\to+\infty$, for some $\bar{x}\in\overline{\R^N\setminus K}$. Furthermore, since the functions $u_{\eps_j}$ converge to $u_0\equiv1$ as $j\to+\infty$ in the sense of~\eqref{uepsk}, we obtain that
$$ 1>\theta>u_{\eps_j}(x_j)\mathop{\longrightarrow}_{j\to+\infty}u_0(\bar{x})=1. $$
This is a contradiction. Therefore, there exists an $\eps_0\in(0,1]$ such that $u_\eps=1$ in $\overline{\R^N\setminus K_\eps}$ for every $\eps\in(0,\eps_0)$ and for every measurable solution $u_\eps:\R^N\setminus K_\eps\to[0,1]$ of~\eqref{eq:ueps} (after identification with its continuous representative). The proof of Theorem~\ref{TH:PERTURB} is thereby complete.
\end{proof}


\vspace{2mm}

\end{document}